\RequirePackage{fix-cm}
\documentclass[smallextended,numbook,runningheads]{svjour3}
\smartqed  
\usepackage{graphicx}
\usepackage{amsmath,amssymb}
\usepackage{graphicx,graphics,color}
\usepackage{booktabs,epstopdf} 
\usepackage{hyperref}
\usepackage{subcaption} 

\journalname{...}
\date{ \phantom{b} \vspace{45mm}\phantom{e}}
\def\makeheadbox{\hspace{-4mm}Version of \today}

\usepackage{color}

\def\half{\tfrac{1}{2}}

\newcommand\bfd{{\mathbf d}}
\newcommand\bfe{{\mathbf e}}
\newcommand\bff{{\mathbf f}}
\newcommand\bfg{{\mathbf g}}

\newcommand\bfu{{\mathbf u}}
\newcommand\bfv{{\mathbf v}}
\newcommand\bfw{{\mathbf w}}
\newcommand\bfx{{\mathbf x}}
\newcommand\bfy{{\mathbf y}}
\newcommand\bfz{{\mathbf z}}
\newcommand\bfA{{\mathbf A}}

\newcommand\bfE{{\mathbf E}}

\newcommand\bfK{{\mathbf K}}

\newcommand\bfM{{\mathbf M}}

\newcommand\andquad{\quad\hbox{ and }\quad}
\newcommand\for{\quad\hbox{ for }\quad}

\renewcommand\d{\hbox{\rm d}}

\newcommand{\Ga}{\Gamma}

\newcommand{\laplace}{\Delta}

\newcommand{\nbg}{\nabla_{\Gamma}}
\newcommand{\nbgh}{\nabla_{\Gamma_h}}
\newcommand{\mat}{\partial^{\bullet}}

\DeclareMathOperator{\diff}{\frac{\d}{\d t}}

\newcommand{\inv}{^{-1}}
\newcommand{\la}{\langle}

\newcommand{\nb}{\nabla}

\newcommand{\pa}{\partial}
\newcommand{\R}{\mathbb{R}}
\newcommand{\ra}{\rangle}

\newcommand{\spn}{\textnormal{span}}
\newcommand{\st}{such that}

\def \t {(t)}

\def \to {\rightarrow}

\usepackage{marginnote}

\newcommand\enodes\xs
\newcommand\nnodes\bfx
\newcommand\rnodes\xs

\newcommand\regmass\bfK
\newcommand\mass\bfM
\newcommand\stiff\bfA

\newcommand{\Xs}{X_h^\ast}

\newcommand{\vs}{\bfv_\ast}
\newcommand{\xs}{\bfx_\ast}

\newcommand{\dotxs}{\dot\bfx_\ast}

\newcommand{\ev}{\bfe_\bfv}
\newcommand{\ex}{\bfe_\bfx}

\newcommand{\dv}{\bfd_\bfv}
\newcommand{\dx}{\bfd_\bfx}

\newcommand{\bbk}{\color{black}}
\newcommand{\ebk}{\color{black}}

\newcommand{\MCF}{\beta \bfA}

\begin{document}

\title{Linearly implicit full discretization of  \\ surface evolution}

\titlerunning{Linearly implicit full discretization of surface evolution}        

\author{Bal\'{a}zs~Kov\'{a}cs \and
        Christian~Lubich 
}

\authorrunning{B.~Kov\'{a}cs and Ch.~Lubich} 

\institute{B.~Kov\'{a}cs, Ch.~Lubich \at
              Mathematisches Institut, Universit\"at T\"{u}bingen,\\
              Auf der Morgenstelle 10, 72076 T\"{u}bingen, Germany \\
              \email{\{kovacs,lubich\}@na.uni-tuebingen.de}
}


\maketitle

\begin{abstract}
Stability and convergence of full discretizations of various surface evolution equations are studied in this paper. The proposed discretization combines a higher-order evolving-surface finite element method (ESFEM) for space discretization with higher-order linearly implicit backward difference formulae (BDF) for time discretization. The stability of the full discretization is studied in the matrix--vector formulation of the numerical method. The geometry of the problem enters into the bounds of the consistency errors, but does not enter into the proof of stability. Numerical examples illustrate the convergence behaviour of the full discretization.

\end{abstract}

\section{Introduction}
In this paper we study full discretizations of  geometric evolution equations using the evolving surface finite element method (ESFEM) for space discretization and linearly implicit backward differentiation formulae (BDF) for time discretization. We consider the situation where the velocity $v(x,t)$ of a point $x$ on an evolving two-dimensional closed surface $\Gamma(t)\subset\R^3$ at time $t$ is determined by one of the following velocity laws, for which finite element semi-discretization in space was studied in \cite{soldriven}:

(i) {\it Regularized mean curvature flow:\/} for $x\in\Ga(t)$,
\begin{equation}\label{v-eq}
	v(x,t) - \alpha \laplace_{\Ga(t)} v(x,t) = -\beta H_{\Ga(t)}(x)\, \nu_{\Ga(t)}(x) +g\bigl(x,t\bigr)\, \nu_{\Ga(t)}(x),
\end{equation}
where $\laplace_{\Ga(t)}$ is the Laplace--Beltrami operator on the surface $\Gamma(t)$, $H_{\Ga(t)}$ is mean curvature, $\nu_{\Ga(t)}$ is the outer normal, $g$ is a smooth real-valued function, and $\alpha>0$ and $\beta\ge 0$ are fixed parameters. This velocity law can be viewed as an elliptically regularized mean curvature flow with an additional driving term in the direction of the normal vector.
In \cite{soldriven} this elliptic regularization allowed us to give a complete stability and convergence analysis of the ESFEM semi-discretization, for finite elements of polynomial degree at least two. In contrast, for pure mean curvature flow (that is, $\alpha=0$), no convergence results  appear to be known for ESFEM on two-dimensional closed surfaces.

(ii)  {\it A dynamic velocity law:\/} for $x\in\Ga(t)$,
\begin{equation} \label{v-dyn}
    \mat v(x,t) + v(x,t) \,\nb_{\Ga(t)} \cdot v (x,t) - \alpha \laplace_{\Ga(t)} v(x,t) = g(x,t) \, \nu_{\Ga(t)}(x),
\end{equation}
where $\mat v$ denotes the material time derivative of $v$ and $\nb_{\Ga} \cdot v$ denotes the surface divergence of $v$;

(iii) The case where the velocity law (i) or (ii)  is coupled to diffusion on the evolving surface, as in \cite{soldriven}.

We note that in all these cases, the considered velocity $v$ is in general not normal to the surface, but contains tangential components.

The rigorous study of the stability and convergence properties of full discretizations obtained by combining the ESFEM with various time discretizations for problems on evolving surfaces was begun in the papers \cite{DziukElliott_fulldiscr} (implicit Euler method), \cite{DziukLubichMansour_rksurf} (implicit Runge--Kutta methods) and \cite{LubichMansourVenkataraman_bdsurf} (BDF methods). These papers studied a linear parabolic equation on a {\it given} moving closed surface $\Gamma(t)$. 
Convergence of full discretizations of that problem using higher-order evolving surface finite elements is studied in \cite{highorder}.
Convergence properties  of full discretizations for quasi- and semilinear parabolic equations on prescribed moving surfaces are studied in \cite{KovacsPower_quasilinear}. For curves instead of two-dimensional surfaces,
convergence of full discretizations of curve-shortening flow coupled to diffusion is studied by Barrett, Deckelnick \& Styles \cite{BDS}.

The main difficulty in proving the convergence of the full discretization of the surface-evolution equation in (i)--(iii) is the proof of stability in the sense of bounding errors in terms of defects in the discrete equations. The proof requires some auxiliary results from \cite{soldriven}, which relate different finite element surfaces. For \eqref{v-eq}, the stability proof just uses the zero-stability of the BDF methods up to order~6. For \eqref{v-dyn}, it is based on energy estimates that become available for BDF methods up to order~5 by  the multiplier technique of Nevanlinna and Odeh \cite{NevanlinnaOdeh}, which in turn is based on the $G$-stability theory of Dahlquist \cite{Dahlquist}. These techniques were originally developed for stiff ordinary differential equations and have recently been used for linear parabolic equations on given moving surfaces in \cite{LubichMansourVenkataraman_bdsurf} and for various quasilinear parabolic problems in \cite{AkrivisLubich_quasilinBDF,AkrivisLiLubich_quasilinBDF,KovacsPower_quasilinear}.

The paper is organized as follows.

In Section~\ref{section: problem} we describe the problem and the numerical methods. We recall the basics of the evolving surface finite element method and give its matrix--vector formulation, and we formulate the linearly implicit BDF time discretization.

In Section~\ref{section: main result} we present the main result for \eqref{v-eq}, which gives optimal-order convergence estimates for the full discretization by ESFEM of polynomial degree at least~$2$ and linearly implicit BDF methods up to order 6. This result is proven in Sections~4 to~7.

Section~\ref{section: aux} contains auxiliary results for the stability analysis of the discretized velocity law \eqref{v-eq}. We collect results from \cite{soldriven} that relate different finite element surfaces to one another. We also include  a new auxiliary result for the linearly implicit BDF time discretization. 

Section~\ref{section: proof of stability} contains the stability analysis, which works with the matrix--vector formulation of the discrete equations. Like the proof of stability of the ESFEM spatial semi-discretization in \cite{soldriven}, it does not use geometric arguments.

Section~\ref{section: consistency} gives estimates for the consistency errors, that is, for the defects on inserting the interpolated exact solution into the discrete equations.

Section~\ref{section: proof completed} proves the convergence result for the full discretization of \eqref{v-eq} by combining the results of the previous sections.

In Section~\ref{section: dynamic} we extend the convergence analysis to the full discretization of the dynamic velocity law \eqref{v-dyn}. This is done for BDF methods up to order~5 using energy estimates based on the Nevanlinna--Odeh multiplier technique.

In Section~\ref{section: coupled} we extend the convergence result for the full discretization to the case where the velocity law \eqref{v-eq} or \eqref{v-dyn} is coupled to diffusion on the evolving surface, as studied in \cite{soldriven} for the semi-discretization. The result is obtained by combining the techniques of \cite{soldriven} and \cite{LubichMansourVenkataraman_bdsurf} with those of Sections~\ref{section: aux} to \ref{section: proof completed} of the present paper.

Section~\ref{section: numerics} presents numerical experiments using quadratic ESFEM that illustrate the numerical behaviour of the proposed full discretization.

We use the notational convention to denote vectors in $\R^3$ by italic letters, but to denote finite element nodal vectors in  $\R^{3N}$ by boldface lowercase letters and finite element mass and stiffness matrices by boldface capitals. All boldface symbols in this paper will thus be related to the matrix--vector formulation of the ESFEM.

\section{Problem formulation and  ESFEM / BDF full discretization}
\label{section: problem}

We use the same setting as in our previous work \cite{soldriven}. We recall basic notions,  but refer to Section~2 of \cite{soldriven} for a more detailed description.

\subsection{Basic notions and notation}
\label{subsection: basic notions}
%
We consider the evolving two-dimensional closed surface $\Gamma(t)\subset\R^3$ as the image
$$
	\Ga(t) = \{ X(q,t) \,:\, q \in \Ga^0 \}
$$
of a regular vector-valued function $X:\Ga^0\times [0,T]\to \R^3$, where $\Ga^0$ is the smooth closed initial surface, and $X(q,0)=q$.
To indicate the dependence of the surface on $X$, we write
$$
	\Ga(t) = \Ga(X(\cdot,t)), \quad\hbox{ or briefly}\quad \Ga(X)
$$
when the time $t$ is clear from the context. The position $X(q,\cdot)$ is related to the {\it velocity} $v(x,t)\in\R^3$ at the point $x=X(q,t)\in\Gamma(t)$  via the ordinary differential equation
\begin{equation}
\label{velocity}
	\partial_t X(q,t)= v(X(q,t),t).
\end{equation}
%
For $x\in\Gamma(t)$ and $0\le t \le T$,
we denote by $\nu_{\Ga(X)}(x)$ the outer normal, by $\nabla_{\Ga(X)}u(x,t)$ the tangential gradient of a real-valued function $u$ on $\Ga(t)$, and by $\laplace_{\Ga(X)} u(x,t)$ the Laplace--Beltrami operator applied to $u$.

\subsection{Weak formulation of the surface-evolution equation}
\label{section: problem and weak form}
%
%

The space discretization is based on the weak formulation of the surface-evolution equation \eqref{v-eq}, which reads as follows:
Find $v(\cdot,t) \in  W^{1,\infty}(\Ga(X(\cdot,t)))^3 $ such that for all test functions $\psi(\cdot,t)  \in H^1(\Ga(X(\cdot,t) ))^3$,
\begin{equation}
\label{weak form}
    \begin{aligned}
        \int_{\Ga(X)} \!\!\!\! v \cdot \psi 
        &\ + \alpha \int_{\Ga(X)} \!\!\!\! \nabla_{\Gamma(X)} v \cdot \nabla_{\Gamma(X)} \psi \\
        &\ + \beta \int_{\Ga(X)} \!\!\!\! \nabla_{\Gamma(X)} X \cdot \nabla_{\Gamma(X)} \psi 
        = \int_{\Ga(X)} \!\! g \,\nu_{\Ga(X)} \cdot \psi ,
    \end{aligned}
\end{equation}
alongside with the ordinary differential equation \eqref{velocity} for the positions $X$ determining the surface $\Gamma(X)$. (More precisely, the term $\nabla_{\Gamma(X)} X$ should read $\nabla_{\Gamma(X)} {\rm id}_{\Gamma(X)}$.)

We assume throughout this paper that the problem \eqref{v-eq} or \eqref{weak form} admits a  unique solution with sufficiently high Sobolev regularity on the time interval $[0,T]$ for the given initial data $X(\cdot,0)$. We assume further that the flow map $X(\cdot,t):\Gamma_0\to \Gamma(t)\subset\R^3$ is non-degenerate for $0\le t \le T$, so that $\Gamma(t)$ is a regular surface.

\subsection{Evolving surface finite elements}
\label{section:ESFEM}

From Section~2.3 of \cite{soldriven} we recall the description of the surface finite element discretization of our problem, which is based on \cite{Dziuk88} and \cite{Demlow2009}. We use simplicial elements and continuous piecewise polynomial basis functions of degree $k$, as defined in \cite[Section 2.5]{Demlow2009}.

We triangulate the given smooth surface $\Gamma^0$ by an admissible family of triangulations $\mathcal{T}_h$ of decreasing maximal element diameter $h$; see \cite{DziukElliott_ESFEM} for the notion of an admissible triangulation, which includes quasi-uniformity and shape regularity. For a momentarily fixed $h$, we denote by $\bfx^0=(x_1^0,\dots,x_N^0)$ the vector in $\R^{3N}$ that collects all $N$ nodes of the triangulation. By piecewise polynomial interpolation of degree $k$, the nodal vector defines an approximate surface $\Gamma_h^0$ that interpolates $\Gamma^0$ in the nodes $x_j^0$. We will evolve the $j$th node in time, denoted $x_j(t)$ with $x_j(0)=x_j^0$, and collect the nodes at time $t$ in a column vector in $\R^{3N}$,
$$
\bfx(t) = (x_1(t),\dots,x_N(t)) \in \R^{3N}.
$$
We just write $\bfx$ for $\bfx(t)$ when the dependence on $t$ is not important.

By piecewise polynomial interpolation on the  plane reference triangle that corresponds to every
 curved triangle of the triangulation, the nodal vector $\bfx$ defines a closed surface denoted by $\Gamma_h[\bfx]$. We can then define finite element {\it basis functions}
$$
\phi_j[\bfx]:\Gamma_h[\bfx]\to\R, \qquad j=1,\dotsc,N,
$$
which have the property that on every triangle their pullback to the reference triangle is polynomial of degree $k$, and which satisfy
\begin{equation*}
    \phi_j[\bfx](x_k) = \delta_{jk} \quad  \text{ for all } j,k = 1,  \dotsc, N .
\end{equation*}
These functions span the finite element space on $\Gamma_h[\bfx]$,
\begin{equation*}
    S_h[\bfx] = S_h(\Gamma_h[\bfx])=\spn\big\{ \phi_1[\bfx], \phi_2[\bfx], \dotsc, \phi_N[\bfx] \big\} .
\end{equation*}
For a finite element function $u_h\in S_h[\bfx]$ the tangential gradient $\nabla_{\Gamma_h[\bfx]}u_h$ is defined piecewise on each element.
We set
$$
	X_h(q_h,t) = \sum_{j=1}^N x_j(t) \, \phi_j[\bfx(0)](q_h), \qquad q_h \in \Gamma_h^0,
$$
which has the properties that $X_h(q_j,t)=x_j(t)$ for $j=1,\dots,N$, that $X_h(q_h,0)=q_h$ for all $q_h\in\Gamma_h^0$, and
$$
	\Gamma_h[\bfx(t)]=\Gamma(X_h(\cdot,t)).
$$
The {\it discrete velocity} $v_h(x,t)\in\R^3$ at a point $x=X_h(q_h,t)\in \Gamma(X_h(\cdot,t))$ is given by
$$
	\partial_t X_h(q_h,t) = v_h(X_h(q_h,t),t).
$$
In view of the transport property of the basis functions  \cite{DziukElliott_ESFEM},
$$
	\frac\d{\d t} \Bigl( \phi_j[\bfx(t)](X_h(q_h,t)) \Bigr) =0 ,
$$
the discrete velocity equals, for $x \in \Gamma_h[\bfx(t)]$,
$$
	v_h(x,t) = \sum_{j=1}^N v_j(t) \, \phi_j[\bfx(t)](x) 
	\qquad \hbox{with } \ v_j(t)=\dot x_j(t),
$$
where the dot denotes the time derivative $\d/\d t$. 
Hence, the nodal vector of the discrete velocity is $\bfv=\dot\bfx$.

%

\subsection{ESFEM spatial semi-discretization of the evolving-surface problem}
\label{subsection:semi-discretization}

The finite element spatial semi-discretization of the problem \eqref{weak form} reads as follows: Find the unknown nodal vector $\bfx(t)\in \R^{3N}$ and the unknown finite element function $v_h(\cdot,t)\in S_h[\bfx(t)]^3$  such that, for all $\psi_h(\cdot,t)\in S_h[\bfx(t)]^3$,
\begin{equation}
\label{uh-vh-equation}
    \begin{aligned}
        \int_{\Ga_h[\bfx]}\!\!\! v_h \cdot \psi_h 
        &\ + \alpha \int_{\Ga_h[\bfx]} \!\!\!\!\! \nabla_{\Ga_h[\bfx]}  v_h \cdot \nabla_{\Ga_h[\bfx]}  \psi_h \\
        &\ + \beta \int_{\Ga_h[\bfx]} \!\!\!\!\! \nabla_{\Ga_h[\bfx]}  X_h \cdot \nabla_{\Ga_h[\bfx]}  \psi_h 
        = \int_{\Ga_h[\bfx]} \!\! g \,\nu_{\Ga_h[\bfx]} \cdot \psi_h,
    \end{aligned}
\end{equation}
and
\begin{equation}\label{xh}
\partial_t X_h(q_h,t) = v_h(X_h(q_h,t),t), \qquad q_h\in\Ga_h^0.
\end{equation}
The initial values for the nodal vector $\bfx$ of the initial positions are taken as the exact initial values at the nodes $x_j^0$ of the triangulation of the given initial surface $\Gamma^0$:
$$
	x_j(0) = x_j^0, \qquad j=1,\dotsc,N.
$$

\subsection{Matrix--vector formulation}\label{subsection:DAE}
We define the surface-dependent mass matrix $\bfM(\bfx)$ and stiffness matrix $\bfA(\bfx)$ on the surface determined by the nodal vector $\bfx$ (cf.~\cite[Section~2.5]{soldriven}):
\begin{equation*}
    \begin{aligned}
        \bfM(\bfx)\vert_{jk} =&\ \int_{\Ga_h[\bfx]} \! \phi_j[\bfx] \phi_k[\bfx] , \\
        \bfA(\bfx)\vert_{jk} =&\ \int_{\Ga_h[\bfx]} \! \nb_{\Ga_h} \phi_j[\bfx] \cdot \nb_{\Ga_h} \phi_k[\bfx] ,
    \end{aligned}
    \qquad (j,k = 1,  \dotsc, N) .
\end{equation*}
We further let  (with the identity matrix $I_3 \in \R^{3 \times 3}$) 
$$
	\bfM^{[3]}(\bfx)= I_3 \otimes \bfM(\bfx) \quad \textnormal{and} \quad \bfA^{[3]}(\bfx) = I_3 \otimes \bfA(\bfx) , 
$$
and then define
\begin{equation}
\label{eq: K matrix def}
    \bfK(\bfx)= \bfM^{[3]}(\bfx) + \alpha \bfA^{[3]}(\bfx) .  
\end{equation}
When no confusion can arise, we write in the following $\bfM(\bfx)$ for $\bfM^{[3]}(\bfx)$, $\bfA(\bfx)$ for $\bfA^{[3]}(\bfx)$ and
$\| \cdot \|_{H^1(\Gamma)}$ for $\| \cdot \|_{H^1(\Gamma)^3}$, etc.

The right-hand side vector $\bfg(\bfx,t)\in\R^{3N}$ is given by
\begin{equation*}
 \begin{aligned}
    \bfg(\bfx,t)\vert_{j+N(\ell-1)} &= \int_{\Ga_h[\bfx]} g(\cdot,t) \,\bigl(\nu_{\Ga_h[\bfx]}\bigr)_\ell \, \phi_j[\bfx],
    \end{aligned}
\end{equation*}
for $j = 1,  \dotsc, N$ and $\ell=1,2,3$.

We then obtain from \eqref{uh-vh-equation}--\eqref{xh} the following system of ordinary differential equations (ODEs) for the nodal vectors $\bfx(t)\in\R^{3N}$:
\begin{equation}
\label{eq: ODE form}
	\bfK(\bfx) \dot \bfx + \MCF(\bfx) \bfx = \bfg(\bfx,t) .
\end{equation}

\subsection{Linearly implicit BDF time discretization}

We apply a $p$-step linearly implicit backward difference formula (BDF) for $p \leq 6$ as a time discretization to the ODE system \eqref{eq: ODE form}. For a step size $\tau>0$, and with $t_n = n \tau \leq T$, we determine the approximation $\bfx^n$ to $\bfx(t_n)$ by the fully discrete system of linear equations
\begin{equation}
\label{BDF}
	\begin{aligned}
		\bfK(\widetilde \bfx^n) \bfv^n + \MCF(\widetilde \bfx^n) \bfx^n = &\ \bfg(\widetilde \bfx^n,t_n) , \\
		\bfv^n =&\ \frac{1}{\tau} \sum_{j=0}^p \delta_j \bfx^{n-j}  ,
	\end{aligned} 
	\qquad\ n \geq p ,
\end{equation}
where the extrapolated position vector $\widetilde \bfx^n$ is defined by
\begin{equation}
\label{eq: extrapolation of u def}
	\widetilde \bfx^n = \sum_{j=0}^{p-1} \gamma_j \bfx^{n-1-j} , \qquad n \geq p .
\end{equation}
The starting values $\bfx^0, \bfx^1, \dotsc, \bfx^{p-1}$ are assumed to be given. They can be precomputed in a way as is usual with multistep methods: using lower-order methods with smaller step sizes or using an implicit Runge--Kutta method. 

The coefficients are given by $\delta(\zeta)=\sum_{j=0}^p \delta_j \zeta^j=\sum_{\ell=1}^p \frac{1}{\ell}(1-\zeta)^\ell$ and $\gamma(\zeta) = \sum_{j=0}^{p-1} \gamma_j \zeta^j = (1 - (1-\zeta)^p)/\zeta$. 
The classical BDF method is known to be zero-stable for $p\leq6$ and to have order $p$; see \cite[Chapter~V]{HairerWannerII}.
This order is retained by the linearly implicit variant using the above coefficients $\gamma_j$; cf.\, \cite{AkrivisLubich_quasilinBDF,AkrivisLiLubich_quasilinBDF}.

We note that the method requires solving a linear system with the symmetric positive definite matrix $\frac{\delta_0}\tau \bfK(\widetilde \bfx^n) + \MCF(\widetilde \bfx^n)$ in the $n$th time step.

From the vectors $\bfx^n =(x_j^n)$ and $\bfv^n = (v_j^n)$ we obtain  position and velocity approximations to $X(\cdot,t_n)$ and $v(\cdot,t_n)$ as
\begin{equation}\label{x-v-approx}
  \begin{split}
	X_h^n(q_h) &= \sum_{j=1}^N x_j^n \, \phi_j[\bfx(0)](q_h) \quad\hbox{ for } q_h \in \Gamma_h^0,\\
        v_h^n(x)      &= \sum_{j=1}^N v_j^n \, \phi_j[\bfx^n](x)  \qquad\hbox{ for } x \in \Gamma_h[\bfx^n].
  \end{split}
\end{equation}

\subsection{Lifts}
\label{section:lifts}

Here we recapitulate \cite[Section~2.6]{soldriven}. 
In the error analysis we need to compare functions on three different surfaces: the {\it exact surface} $\Gamma(t)=\Gamma(X(\cdot,t))$, the {\it discrete surface} $\Gamma_h(t)=\Gamma_h[\bfx(t)]$, and the {\it interpolated surface} $\Gamma_h^*(t)=\Gamma_h[\xs(t)]$, where $\xs(t)$ is the nodal vector collecting the grid points $x_{*,j}(t)=X(q_j,t)$ on the exact surface. In the following definitions we omit the argument $t$ in the notation. 

For a finite element function $w_h:\Gamma_h\to\R^m$ ($m=1$ or 3) on the discrete surface, with nodal values $w_j$, we denote by $\widehat w_h:\Gamma_h^*\to\R^m$ the finite element function  on the interpolated surface that has the same nodal values:
$$
\widehat w_h  = \sum_{j=1}^N w_j \phi_j[\xs].
$$
 The transition between the interpolated surface and the exact surface is  done by the
 \emph{lift operator}, which was introduced  for linear surface approximations  in \cite{Dziuk88};  see also \cite{DziukElliott_ESFEM,DziukElliott_L2}.
 Higher-order generalizations have been studied in \cite{Demlow2009}.  The lift operator $l$ maps a function on the interpolated surface $\Gamma_h^*$ to a function on the exact surface $\Gamma$, provided that $\Gamma_h^*$ is sufficiently close to~$\Gamma$.

The exact regular surface $\Ga(X(\cdot,t))$ can be represented by a (sufficiently smooth) signed distance function $d : \R^3 \times [0,T] \to \R$, cf.~\cite[Section~2.1]{DziukElliott_ESFEM}, such that
	$\Ga(X(\cdot,t)) = \big\{ x\in \R^3 \mid d(x,t) = 0 \big\} \subset \R^3$ .
Using this distance function,  the lift of a continuous function $\eta_h \colon \Ga_h^* \to \R^m$ is defined as
\begin{equation*}
    \eta_{h}^{l}(y) := \eta_h(x), \qquad x\in\Ga_h^*,
\end{equation*}
where for every $x\in \Ga_h^*$ the point $y=y(x)\in\Ga$ is uniquely defined via
    $y = x - \nu(y) d(x)$.

We denote the composed lift $L$ from finite element functions on $\Gamma_h$ to functions on $\Gamma$ via $\Gamma_h^*$ by
$$
w_h^L = (\widehat w_h)^l.
$$

\section{Statement of the main result: fully discrete error bound}
\label{section: main result}

We formulate the main result of this paper, which yields optimal-order error bounds for the ESFEM /  BDF full discretization of  the surface-evolution equation \eqref{v-eq}, for finite elements of polynomial degree $k\ge 2$ and BDF methods of order $p\le 6$. We denote by $\Gamma(t_n)=\Gamma(X(\cdot,t_n))$ the exact surface and by $\Gamma_h^n=\Gamma(X_h^n)=\Gamma_h[\bfx^n]$ the discrete surface at time $t_n$. For the lifted position function we introduce the notation
$$
(x_h^n)^L (x)  =  (X_h^n)^L(q) \in \Gamma_h^n \qquad\hbox{for}\quad x=X(q,t_n)\in\Gamma(t_n).
$$

\begin{theorem}
\label{theorem: main}
    Consider the ESFEM / BDF linearly implicit full discretization \eqref{BDF} of the surface-evolution equation \eqref{v-eq}, using finite elements of polynomial degree~$k\ge 2$ and BDF methods of order $p\le 6$.
    We assume quasi-uniform admissible triangulations of the initial surface and initial values chosen by finite element interpolation of the initial data for $X$.
    Suppose that the problem admits an exact solution $(X,v)$ that is sufficiently smooth (say, of class  $C([0,T],H^{k+1})\cap C^{p+1}([0,T],W^{1,\infty})$) on the time interval $0\le t \le T$,  and that the flow map $X(\cdot,t):\Gamma_0\to \Gamma(t)\subset\R^3$ is non-degenerate for $0\le t \le T$, so that $\Gamma(t)$ is a regular surface. Suppose further that the starting values are sufficiently accurate:
  $$
  \| (X_h^i)^L - X(\cdot,i\tau) \|_{H^1(\Ga^0)^3} \le C_0  (h^k+\tau^{p}), \qquad i=0,1,\dots,p-1.
  $$
    Then, there exist $h_0 >0$, $\tau_0>0$ and $c_0>0$ such that for all mesh widths $h \leq h_0$ and step sizes $\tau\le\tau_0$ satisfying the mild stepsize restriction
    $$
    \tau^p \le c_0 h,
    $$
    the following error bounds hold over the exact surface $\Ga(t_n)=\Ga(X(\cdot,t_n))$
uniformly  for $0\le t_n=n\tau \le T$:
    \begin{equation*}
 	\begin{aligned}
   		\|(x_h^n)^{L} - \mathrm{id}_{\Gamma(t_n)}\|_{H^1(\Ga(t_n))^3} &\leq C(h^k+\tau^p),\\
		   \|(v_h^n)^{L} - v(\cdot,t_n)\|_{H^1(\Ga(t_n))^3}  &\leq C(h^k+\tau^p).
 	\end{aligned}
    \end{equation*}
 The constant $C$ is independent of  $h$ and $\tau$ and $n$ with $n\tau\le T$,  but depends on bounds of higher derivatives of the solution $(X,v)$,  and on the length $T$ of the time interval.
\end{theorem}

We note that the first error bound is equivalent to
$$
	\| (X_h^n)^L - X(\cdot,t_n) \|_{H^1(\Ga^0)^3} \le C' (h^k + \tau^{p}) ,
$$
and we mention that the remarks after Theorem 3.1 in \cite{soldriven} (the convergence theorem of the ESFEM semi-discretization) apply also to the fully discretized situation considered here.

The proof of Theorem~\ref{theorem: main} is given in the course of the next four sections.

\section{Preparation: Estimates relating different surfaces}
\label{section: aux}

%

%

In our previous work \cite[Section~4]{soldriven} we have shown some auxiliary results relating different finite element surfaces, which we recapitulate here.

The  finite element matrices of Section~\ref{subsection:DAE} induce discrete versions of Sobolev norms. For any $\bfw=(w_j) \in \R^N$ with corresponding finite element function $w_h= \sum_{j=1}^N w_j \phi_j[\bfx] \in S_h[\bfx]$ we note
\begin{align} \label{M-L2}
   &  \|\bfw\|_{\bfM(\bfx)}^{2} = \bfw^T \bfM(\bfx) \bfw = \|w_h\|_{L^2(\Ga_h[\bfx])}^2, \\
   \label{A-H1}
   &  \|\bfw\|_{\bfA(\bfx)}^{2} = \bfw^T \bfA(\bfx) \bfw = \|\nb_{\Ga_h[\bfx]} w_h\|_{L^2(\Ga_h[\bfx])}^2.
\end{align}

We use the following setting.
Let $\bfx,\bfy \in \R^{3 N}$ be two nodal vectors defining discrete surfaces $\Gamma_h[\bfx]$ and $\Gamma_h[\bfy]$, respectively. We let $\bfe= (e_j)=\bfx-\bfy \in \R^{  3  N}$. For  $\theta\in[0,1]$, we consider the intermediate surface $\Gamma_h^\theta=\Gamma_h[\bfy+\theta\bfe]$ and the corresponding finite element functions given as
$$
e_h^\theta=\sum_{j=1}^N e_j \phi_j[\bfy+\theta\bfe]
$$
and in the same way, for any vectors $\bfw,\bfz \in \R^N$,
$$
    w_h^\theta=\sum_{j=1}^N w_j \phi_j[\bfy+\theta\bfe] \andquad z_h^\theta=\sum_{j=1}^N z_j \phi_j[\bfy+\theta\bfe] .
$$
The following lemma collects results from  \cite[Section~4]{soldriven}.
\begin{lemma}
\label{lemma: technicals}
    (i) \label{matrix differences}   In the above setting the following identities hold:
    \begin{align*}
        \bfw^T (\bfM(\bfx)-\bfM(\bfy)) \bfz =&\ \int_0^1 \int_{\Ga_h^\theta} w_h^\theta (\nabla_{\Ga_h^\theta} \cdot e_h^\theta) z_h^\theta \; \d\theta, \\
        \bfw^T (\bfA(\bfx)-\bfA(\bfy)) \bfz =&\ \int_0^1 \int_{\Ga_h^\theta} \nb_{\Ga_h^\theta} w_h^\theta \cdot (D_{\Ga_h^\theta} e_h^\theta)\nb_{\Ga_h^\theta}  z_h^\theta \; \d\theta ,
    \end{align*}
    with
    $D_{\Ga_h^\theta} e_h^\theta =  \textnormal{trace}(E) I_3 - (E+E^T)$ for $E=\nabla_{\Ga_h^\theta} e_h^\theta \in \R^{3\times 3}$.

	(ii) \label{lemma:cond-equiv}
	If
		$\| \nabla_{\Gamma_h^\theta} \cdot e_h^\theta \|_{L^\infty(\Gamma_h^\theta)} \leq \mu $
		and 
		$\| D_{\Gamma_h^\theta} e_h^\theta \|_{L^\infty(\Gamma_h^\theta)} \leq \rho \ebk$
for $0 \leq \theta \leq 1$,  then
	$
	 \| \bfw \|_{\bfM(\bfy+\theta\bfe)}  \le e^{\mu/2} \, \| \bfw \|_{\bfM(\bfy)}
	 $
	and 
	$
	 \| \bfw \|_{\bfA(\bfy+\theta\bfe)}  \le e^{\rho \ebk/2} \, \| \bfw \|_{\bfA(\bfy)}
	$.	
	
	(iii) \label{lemma:theta-independence} If 
	$
	\| \nabla_{\Gamma_h[\bfy]} e_h^0 \|_{L^\infty(\Gamma_h[\bfy])} \le \frac12,
	$
	then, for $0\le\theta\le 1$, the function $w_h^\theta=\sum_{j=1}^N w_j \phi_j[\bfy+\theta\bfe]$ on $\Gamma_h^\theta=\Gamma_h[\bfy+\theta\bfe]$ is bounded by
	$$
	\| \nabla_{\Gamma_h^\theta} w_h^\theta \|_{L^p(\Gamma_h^\theta)} \le c_p \, \| \nabla_{\Gamma_h^0} w_h^0 \|_{L^p(\Gamma_h^0)}
	\for 1\le p \le \infty,
	$$
	where $c_p$ depends only on $p$ (we have $c_\infty=2$).

	(iv) \label{lemma: normal vector perturbation}
	Let $y_h^\theta\in\Gamma_h^\theta$ be defined as $y_h^\theta=\sum_{j=1}^N (y_j +\theta e_j)\phi_j[\bfy](q_h)$ for $q_h\in \Gamma_h[\bfy]$. 
	If 
	$
	\| \nabla_{\Gamma_h[\bfy]} e_h^0 \|_{L^\infty(\Gamma_h[\bfy])} \le \frac12,
	$
	then the corresponding unit normal vectors differ by no more than
	 $$
	 |\nu_{\Gamma_h^\theta}(y_h^\theta) - \nu_{\Gamma_h^0}(y_h^0)| \le C\theta  | \nabla_{\Gamma_h^0} e_h^0(y_h^0) | ,
	 $$
	 where $C$ is independent of $h$ and of $q_h\in \Gamma_h[\bfy]$.
%

\end{lemma}

The following result is shown in Lemma~4.1 of \cite{DziukLubichMansour_rksurf}.

\begin{lemma}
\label{lemma: matrix derivatives}
    Let $\Gamma(t)=\Gamma(X(\cdot,t))$, $t \in [0,T]$, be a smoothly evolving family of smooth closed surfaces, and
    let the vector $\xs(t) \in \R^{3N}$ collect the nodes $x_j^*(t)=X(q_j,t)$.  Then, for $0\le s, t \le T$ and for all $\bfw,\bfz\in \R^N$,
    \begin{align*}
        \bfw^T \bigl(\bfM(\xs(t))  - \bfM(\xs(s))\bigr)\bfz \leq&\ C  (t-s) \|\bfw\|_{\bfM(\xs(t))}\|\bfz\|_{\bfM(\xs(t))} , \\
        \bfw^T \bigl(\bfA(\xs(t))  - \bfA(\xs(s))\bigr)\bfz \leq&\ C  (t-s) \|\bfw\|_{\bfA(\xs(t))}\|\bfz\|_{\bfA(\xs(t))}    
    \end{align*}
    and the norms for different times are uniformly equivalent for $0\le s, t \le T$:
    $$
    \| \bfw \|_{\bfM(\xs(t))} \le C \| \bfw \|_{\bfM(\xs(s))}, \quad\ 
     \| \bfw \|_{\bfA(\xs(t))} \le C \| \bfw \|_{\bfA(\xs(s))}.
     $$
    The constant $C$ depends only on a bound of the $W^{1,\infty}\!$ norm of the surface velocity.
\end{lemma}

%
%
%
%
%
%
We also need a result which compares the finite element surfaces with exact and extrapolated nodes. 
\begin{lemma}
	\label{lemma: matrix identity for extrapolation}
	Let $\Gamma(t)=\Gamma(X(\cdot,t))$, $t \in [0,T]$, be a smoothly evolving family of smooth closed surfaces. We denote the nodal vectors of exact solution values by $\bfx_*^n =\bfx_*(t_n)$ and of the extrapolated values by $\widetilde \bfx_*^n = \sum_{j=0}^{p-1} \gamma_j \bfx_*^{n-1-j}$.
 Then, the following estimates hold for all $\bfw,\bfz\in\R^N$:
	\begin{align*}
		\bfw^T (\bfM(\widetilde\bfx_\ast^n)-\bfM(\xs^n)) \bfz 
		&\leq  C\tau^p \,\|\bfw\|_{\bfM(\xs^n)} \|\bfz\|_{\bfM(\xs^n)} , \\
		\bfw^T (\bfA(\widetilde\bfx_\ast^n)-\bfA(\xs^n)) \bfz 
		&\leq  C\tau^p \, \|\bfw\|_{\bfA(\xs^n)} \|\bfz\|_{\bfA(\xs^n)} ,
	\end{align*}
	where $C$ is independent of $h$, $\tau$ and $n$ with $0\le n\tau\le T$.
\end{lemma}
\begin{proof}
For the extrapolated value $\widetilde X(q,t) = \sum_{j=0}^{p-1} \gamma_j X(q,  t-(j+1)\tau)$, we use the error formula with Peano kernel representation, \bbk see e.g.~\cite[Section~3.2.6]{Gautschi}, \ebk
\begin{equation}
\label{peano}
	\widetilde X(q,t)  - X(q,t) = \tau^p \int_0^p \kappa_p(\lambda) \, \partial_t^{p+1} X(q,t-\lambda\tau)\, \d\lambda
\end{equation}
with a bounded Peano kernel $\kappa_p$. We note that we have 
$$
\widetilde x_{*,j}^n - x_{*,j}^n = \widetilde X(q_j,t_n)  - X(q_j,t_n).
$$
Since $X$ is assumed smooth, we obtain from the above error formula that  for $0 \leq \theta \leq 1$,  the finite element function ${\widetilde e}_h^{n,\theta}$ in $S_h(\Ga_h^{\theta})$ with the nodal vector $\widetilde\bfx_\ast^n - \xs^n$, for $\Ga_h^{\theta} = \Ga_h[\xs^n+\theta (\widetilde\bfx_\ast^n - \xs^n)]$, has a gradient bounded in the maximum norm by  $c\tau^p$, where $c$ is independent of $\tau$ and $h$.
%
So we have  the bound
$$
	\|\nabla_{\Ga_h[\xs^n]} \cdot {\widetilde e}_h^{n,0} \|_{L^\infty(\Ga_h[\xs^n])} \le c\tau^p.
$$
Together with Lemma~\ref{lemma: technicals} and an $L^2 - L^\infty - L^2$ estimate, we thus obtain
\begin{align*}
	\bfw^T (\bfM(\widetilde\bfx_\ast^n)-\bfM(\xs^n)) \bfz 
	=&\ \int_0^1 \int_{\Gamma_h^{n,\theta}} w_h^\theta (\nabla_{\Gamma_h^{n,\theta}} \cdot {\widetilde e}_h^{n,\theta} ) z_h^\theta \d \theta \\
	\leq &\ \int_0^1 \|w_h^\theta\|_{L^2(\Gamma_h^{n,\theta})}
	 \|\nabla_{\Gamma_h^{n,\theta}} \cdot {\widetilde e}_h^{n,\theta} \|_{L^\infty(\Gamma_h^{n,\theta})} 
	 \|z_h^\theta\|_{L^2(\Gamma_h^{n,\theta})} \d \theta \\
	 \leq &\ c\tau^p
	 \|w_h^0\|_{L^2(\Ga_h^{0,n})}
	 \|z_h^0\|_{L^2(\Ga_h^{0,n})} \\ 
	 \leq &\ c\tau^p \|\bfw\|_{\bfM(\xs^n)} \|\bfz\|_{\bfM(\xs^n)} .
\end{align*}
The second estimate is proved in the same way. \qed
\end{proof}

The above lemma immediately implies the following norm equivalence, for sufficiently small step size $\tau$,
\begin{equation}
\label{eq: extrap norm equivalence}
	\half \|\bfw\|_{\bfK(\xs^n)}^2 
	\leq \|\bfw\|_{\bfK(\widetilde\bfx_\ast^n)}^2 
	\leq \tfrac{3}{2} \|\bfw\|_{\bfK(\xs^n)}^2 .
\end{equation}

\section{Stability}
\label{section: proof of stability}

We denote by
$$
	\xs(t)=\bigl( x_{*,j}(t)\bigr)\in \R^{3N}\quad\hbox{with}\quad x_{*,j}(t)=X(q_j,t) , \qquad (j=1,\dots,N)
$$
the nodal vector of the \emph{exact} positions on the surface $\Gamma(X(\cdot,t))$. This defines a discrete surface $\Gamma_h[\xs(t)]$ that interpolates the exact surface $\Gamma(X(\cdot,t))$.

We consider the interpolated  exact velocity
$$
	v_{*,h}(\cdot, t)=\sum_{j=1}^N v_{*,j}(t) \phi_j[\xs(t)] \quad\ \hbox{ with }\quad\ v_{*,j}(t)=\dot x_{*,j}(t),
$$
with the corresponding nodal vector
$$
	\vs(t)=\bigl( v_{*,j}(t)\bigr)= \dotxs(t) \in \R^{3N}.
$$
We write
$$
\bfx_*^n=\bfx_*(t_n), \ \ \bfv_*^n=\bfv_*(t_n).
$$
The errors of the numerical solution values $\bfx^n$ and $\bfv^n$ are marked with their respective subscript, hence are denoted by
\begin{align*}
	\ev^n = \bfv^n - \vs^n, \qquad \ex^n = \bfx^n - \xs^n .
\end{align*}

\subsection{Error equations}
The nodal vectors of the exact solution satisfy the equations of the linearly implicit BDF method only up to defects $\dv^n$ and $\dx^n$ that, for $n \geq p$, are defined by the equations
\newcommand{\tildexs}{\widetilde{\bfx}_{\ast}}
\newcommand{\tildex}{\widetilde\bfx}
\newcommand{\tildeex}{\widetilde{\bfe}_{\bfx}}
\begin{equation}
\label{eq: BDF for exact sol}
	\begin{aligned}
		\bfK(\tildexs^n) \vs^n + \MCF(\tildexs^n) \xs^n =&\ \bfg(\tildexs^n,t_n) + \bfM(\xs^n)\dv^n , \\
		\frac{1}{\tau} \sum_{j=0}^p \delta_j \xs^{n-j} =&\ \vs^n + \dx^n .
	\end{aligned} 
\end{equation}

We subtract \eqref{eq: BDF for exact sol} from \eqref{BDF} to obtain the error equations
\begin{equation}
\label{eq: error equations for BDF}
	\begin{aligned}
		\bfK(\tildexs^n) \ev^n + &\ \MCF(\tildexs^n) \ex^n \\
		=&\ - \big(\bfK(\widetilde \bfx^n) - \bfK(\tildexs^n) \big) \ev^n 
		- \big(\bfK(\widetilde \bfx^n) - \bfK(\tildexs^n) \big) \vs^n \\
		&\ - \beta \big(\bfA(\widetilde \bfx^n) - \bfA(\tildexs^n) \big) \ex^n 
		- \beta \big(\bfA(\widetilde \bfx^n) - \bfA(\tildexs^n) \big) \xs^n \\
		&\ + \bfg(\widetilde \bfx^n,t_n) - \bfg(\tildexs^n,t_n) - \bfM(\tildexs^n)\dv^n , \\
		\frac{1}{\tau} \sum_{j=0}^p \delta_j \ex^{n-j}  &=\ \ev^n - \dx^n  .
	\end{aligned}
\end{equation}

\subsection{Stability bound}

We recall that the matrix $\bfK(\xs)$ defines a norm which is equivalent to the $H^1$ norm on $\Ga_h[\xs]$. The defect  $\bfd_\bfv\in \R^{3N}$ will be measured in the dual norm defined by 
$$
 \|\bfd\|_{\star,\xs}^2 := \bfd^T \bfM(\xs)\bfK(\xs)\inv \bfM(\xs)\bfd ,
$$
which is such that for the finite element function $d_h\in S_h[\bfx^*]^3$ with nodal vector $\bfd$ we have, from \cite[Proof of Theorem 5.1]{LubichMansourVenkataraman_bdsurf} or \cite[Formula (5.5)]{soldriven},
\begin{equation}
\label{dual-norm-h}
  \|\bfd\|_{\star,\xs} = \|d_h\|_{H_h\inv(\Gamma_h[\bfx^*])} := \sup_{0\neq \psi_h \in S_h[\bfx^*]^3} \frac{\int_{\Gamma_h[\bfx^*]} d_h \cdot  \psi_h}{\|\psi_h\|_{H^1(\Gamma_h[\bfx^*])^3}}.
\end{equation}
In these norms we have the following stability result.

\begin{proposition}
\label{propostion: stability - regularised velocity law}
	Suppose that the defects of the $p$-step linearly implicit BDF method are bounded as follows, with a sufficiently small $\vartheta > 0$ (that is independent of $h$ and $\tau$ and $n$): for $n\ge p$ with $n\tau \le T$,
	\begin{equation}
	\label{eq: assume small defects}
		\|\dx^n\|_{\bfK(\xs^k)} \leq \vartheta h
		\andquad
		\|\dv^n\|_{\star,\xs^k} \leq \vartheta h \quad \textnormal{ for } \, k \tau \leq T .
	\end{equation} 
	Further, assume that the initial values are chosen such that	
	\begin{equation}
	\label{eq: small initial errors}
		\|\ex^k\|_{\bfK(\xs^k)} \leq \vartheta h
		\andquad
		\|\ev^k\|_{\bfK(\xs^k)} \leq \vartheta h \quad \textnormal{for } \, k=0,\dotsc,p-1.
	\end{equation}
	Then,
	the following error bounds hold, for $n\ge p$ such that $  n \tau \leq T$, 
	\begin{equation}
	\label{eq: error estimate}
		\begin{aligned}
			\|\ex^{n}\|_{\bfK(\xs^{n})}^2 
			\leq &\ C \tau \sum_{j=p}^{n} \Big( \|\dx^j\|_{\bfK(\xs^j)}^2 + \|\dv^j\|_{\star,\xs^j}^2 \Big) + C\sum_{i=0}^{p-1} \|\ex^i\|_{\bfK(\xs^i)}^2 , \\
			\|\ev^n\|_{\bfK(\xs^n)}^2 
			\leq &\ C \tau \sum_{j=p}^n \Big( \|\dx^j\|_{\bfK(\xs^j)}^2 + \|\dv^j\|_{\star,\xs^j}^2 \Big) + C \|\dv^n\|_{\star,\xs^n}^2 + C\sum_{i=0}^{p-1} \|\ex^i\|_{\bfK(\xs^i)}^2,
		\end{aligned}
	\end{equation}
	where  $C$ is independent of $h$, $\tau$ and $n$ with $n\tau\le T$, but depends on~$T$.
\end{proposition}

In Section~\ref{section: consistency} we will show that the defects obtained on inserting the exact solution values into the BDF scheme satisfy the bounds 
$$
\|\dx^n\|_{\bfK(\xs^n)} \le C(h^k+\tau^p),\quad\ 
\|\dv^n\|_{\star,\xs^n}  \le C(h^k+\tau^p).
$$
Hence, condition \eqref{eq: assume small defects} is satisfied under the mild stepsize restriction
\begin{equation}\label{stepsize-restriction}
\tau^p \le c_0 h
\end{equation}
for a sufficiently small $c_0$ that is independent of $h$ and $\tau$.
We note that the error functions $e_x^n,e_v^n \in S_h[\xs^n]^3$ with nodal vectors $\ex^{n}$ and $\ev^{n}$, respectively, are then bounded by
\begin{equation*}
  \begin{aligned}
	&\| e_x^n \|_{H^1(\Ga_h[\xs^n])} \leq C(h^k+\tau^p), \\ 
	&\| e_v^n \|_{H^1(\Ga_h[\xs^n])} \leq C(h^k+\tau^p) , 
  \end{aligned}
	\qquad \textnormal{ for }  n \tau \leq T.
\end{equation*}

\begin{proof}
The proof is based on energy estimates for the matrix--vector formulation of the error equations \eqref{eq: error equations for BDF} and relies on the results of Section~\ref{section: aux}. In the proof, $c$ will be a generic constant independent of $h$ and $\tau$ and $n$ with $n\tau\le T$, which assumes different values on different occurrences.
For many estimates we use similar techniques of proof as for the corresponding time-continuous results in \cite{soldriven}. However, to keep the paper fairly self-contained we include some detailed arguments.

\smallskip
In view of the condition in (iii) of Lemma~\ref{lemma: technicals}
 for $\bfy=\tildexs^n$ and $\bfx=\tildex^n$, we need to control the $W^{1,\infty}$ norm of the position error $\widetilde{e}_x^n$. 
Let us assume that the error estimate \eqref{eq: error estimate} holds for $ p, \dotsc, n-1$.
Then, using an inverse inequality and the norm equivalence \eqref{eq: extrap norm equivalence} and the definition of $\tildeex^n$ (cf.~\eqref{eq: extrapolation of u def}), we obtain
\begin{equation}
\label{eq: assumed bounds}
	\begin{aligned}
		\|\nabla_{\Ga_h[\tildexs^{n}]} \widetilde{e}_x^n\|_{L^\infty(\Ga_h[\tildexs^{n}])}
		\leq &\ c h^{-1} \|\nabla_{\Ga_h[\tildexs^{n}]} \widetilde{e}_x^n\|_{L^2(\Ga_h[\tildexs^{n}])} \\
		\leq &\ c h^{-1} \| \tildeex^n \|_{\bfK(\tildexs^{n})} \leq c h^{-1} \| \tildeex^n \|_{\bfK(\xs^{n})} \\
		\leq &\ c h^{-1} \sum_{j=1}^p \| \ex^{n-j} \|_{\bfK(\xs^{n})} \\
		\leq &\ c h^{-1}  \cdot c\vartheta h \le c\vartheta,
	\end{aligned}
\end{equation}
where the last but one estimate follows from \eqref{eq: error estimate} for the past, and the assumption on small defects \eqref{eq: assume small defects}. For sufficiently small $\vartheta$, we are thus in the position to use the bounds given in Lemma~\ref{lemma: technicals}.
	
\smallskip
We estimate the two error equations \eqref{eq: error equations for BDF} separately, and then combine them to yield the final estimate. 

(a) \textit{Estimates for the velocity law.} 
By testing the first line of the error equations \eqref{eq: error equations for BDF} with $\ev^n$ we obtain
\begin{align*}
	&\ \half \|\ev^n\|_{\bfK(\xs^n)}^2 \leq \|\ev^n\|_{\bfK(\tildexs^n)}^2 \\
	=&\ - (\ev^n)^T \big(\bfK(\widetilde \bfx^n) - \bfK(\tildexs^n) \big) \vs^n 
	- (\ev^n)^T \big(\bfK(\widetilde \bfx^n) - \bfK(\tildexs^n) \big) \ev^n \\
	&\ - \beta (\ev^n)^T \big(\bfA(\widetilde \bfx^n) - \bfA(\tildexs^n) \big) \xs^n 
	- \beta (\ev^n)^T \big(\bfA(\widetilde \bfx^n) - \bfA(\tildexs^n) \big) \ex^n \\
	&\ + (\ev^n)^T \big(\bfg(\widetilde \bfx^n,t_n) - \bfg(\tildexs^n,t_n)\big) - \beta (\ev^n)^T \bfA(\tildexs^n) \ex^n - (\ev^n)^T \bfM(\tildexs^n)\dv^n ,
\end{align*}
where the inequality follows from \eqref{eq: extrap norm equivalence}.
To bound the right-hand side, we use arguments of the proof of Proposition~10.1 (and that of Proposition~5.1) of \cite{soldriven}, using the results of Lemma~\ref{lemma: technicals}. 

\newcommand{\tildeexth}{\widetilde{e}_{x}}

\noindent (i) For $0 \leq \theta \leq 1$, we denote $\Gamma_h^{n,\theta} = \Gamma_h[\tildexs^n + \theta\tildeex^n]$, where $\tildeex^n = \tildex^n - \tildexs^n = \sum_{j=0}^{p-1} \gamma_j \ex^{n-p+j}$. We denote the finite element functions in $S_h(\Gamma_h^{n,\theta})^3$
with nodal vectors $\tildeex^n$, $\ev^n$ and $\vs^n$ by $\tildeexth^{n,\theta}$,$e_v^{n,\theta}$ and $v_{*}^{n,\theta}$, respectively.
The definition \eqref{eq: extrapolation of u def} and Lemma~\ref{lemma: technicals} then give us
\begin{align*}
	&\ (\ev^n)^T \big(\bfK(\tildex^n) - \bfK(\tildexs^n) \big) \vs^n 
	= \ \int_0^1 \!\! \int_{\Gamma_h^{n,\theta}} \!\!\! e_v^{n,\theta} \cdot \bigl( \nabla_{\Gamma_h^{n,\theta}}\cdot  \tildeexth^{n,\theta} \bigr) v_{*}^{n,\theta} \, \d\theta 
	\\ &\qquad\qquad \qquad
	+ \alpha \int_0^1 \!\! \int_{\Gamma_h^{n,\theta}} \!\!\! \nabla_{\Gamma_h^{n,\theta}} e_v^{n,\theta} \cdot \bigl( D_{\Gamma_h^{n,\theta}} \tildeexth^{n,\theta} \bigr) \nabla_{\Gamma_h^{n,\theta}} v_{*}^{n,\theta} \, \d\theta .
\end{align*}
Using the Cauchy--Schwarz inequality, we estimate the integral with the product of the $L^2-L^2-L^\infty$ norms of the three factors. We thus have
\begin{align*}
	&(\ev^n)^T \big(\bfK(\tildex^n) - \bfK(\tildexs^n) \big) \vs^n \\
	&\leq \ \int_0^1 \| e_v^{n,\theta} \|_{L^2(\Gamma_h^{n,\theta})} \,
	\| \nabla_{\Gamma_h^{n,\theta}}\cdot \tildeexth^{n,\theta} \|_{L^2(\Gamma_h^{n,\theta})} \,
	\| v_{*}^{n,\theta} \|_{L^{\infty}(\Gamma_h^{n,\theta})}\, \d\theta \\ 
	&\ \ \ + \alpha \int_0^1 \| \nabla_{\Gamma_h^{n,\theta}} e_v^{n,\theta} \|_{L^2(\Gamma_h^{n,\theta}) }\, \| D_{\Gamma_h^{n,\theta}} \tildeexth^{n,\theta} \|_{L^2(\Gamma_h^{n,\theta})} \,
	\| \nabla_{\Gamma_h^{n,\theta}} v_{*}^{n,\theta} \|_{L^{\infty}(\Gamma_h^{n,\theta})}\, \d\theta
	\\
	&\leq \ c \int_0^1 \| e_v^{n,\theta} \|_{H^1(\Gamma_h^{n,\theta})} \, 
	\| \tildeexth^{n,\theta} \|_{H^1(\Gamma_h^{n,\theta})} \,
	\| v_{*}^{n,\theta} \|_{W^{1,\infty}(\Gamma_h^{n,\theta})}\, \d\theta.
\end{align*}
By \eqref{eq: assumed bounds} and Lemma~\ref{lemma:theta-independence}, this is bounded by
\begin{align*}
	&(\ev^n)^T \big(\bfK(\tildex^n) - \bfK(\tildexs^n) \big) \vs^n \\ 
	&\leq \ c \| e_v^n \|_{H^1(\Gamma_h[\tildexs^n])} \, 
	\| \tildeexth^{n} \|_{H^1(\Gamma_h[\tildexs^n])} \,
	\| v_*^n \|_{W^{1,\infty}(\Gamma_h[\tildexs^n])} ,
\end{align*}
where the last factor is bounded independently of $h$ and $\tau$.
By Young's inequality, we thus obtain
\begin{align*}
	(\ev^n)^T \big(\bfK(\tildex^n) - \bfK(\tildexs^n) \big) \vs^n 
	\leq & \ \tfrac{1}{48}\| e_v^n \|_{H^1(\Gamma_h[\tildexs^n])}^2 
	+ c \sum_{j=1}^{p} \| e_x^{n-j} \|_{H^1(\Gamma_h[\tildexs^n])}^2 \\
	= &\  \tfrac{1}{48} \|\ev^n\|_{\bfK(\tildexs^n)}^2 
	+ c \sum_{j=1}^{p}  \|\ex^{n-j}\|_{\bfK(\tildexs^n)}^2 \\
	\leq &\  \tfrac{1}{24} \|\ev^n\|_{\bfK(\xs^n)}^2 
	+ c \sum_{j=1}^{p}  \|\ex^{n-j}\|_{\bfK(\xs^n)}^2 ,
\end{align*}
where the last inequality follows from the norm equivalence \eqref{eq: extrap norm equivalence}.

\noindent (ii) Similarly, estimating the three factors in the integrals by $L^2-L^\infty-L^2$, we obtain
\begin{align*}
	(\ev^n)^T \big(\bfK(\widetilde \bfx^n) - \bfK(\tildexs^n) \big) \ev^n 
	\leq &\ c \| e_v^n \|_{L^2(\Gamma_h[\tildexs^n])}^2 \, 
	\| \nbgh \cdot \tildeexth^{n} \|_{L^{\infty}(\Gamma_h[\tildexs^n])} \\
	&\ + c \| \nbgh e_v^n \|_{L^2(\Gamma_h[\tildexs^n])}^2 \, 
	\| D_{\Ga_h}\tildeexth^{n}  \|_{L^{\infty}(\Gamma_h[\tildexs^n])} \\
	\leq &\ c\vartheta \|\ev^n\|_{\bfK(\xs^n)}^2 \leq \tfrac{1}{24} \|\ev^n\|_{\bfK(\xs^n)}^2,
\end{align*}
where  we used the estimate \eqref{eq: assumed bounds} in the last but one  inequality.

\noindent (iii)--(iv) The estimates involving the mean curvature term $\beta \bfA$ (in view of \eqref{eq: K matrix def}) can be shown analogously as (i) and (ii):
\begin{align*}
	& (\ev^n)^T \big(\bfA(\widetilde \bfx^n) - \bfA(\tildexs^n) \big) \xs^n 
	+ (\ev^n)^T \big(\bfA(\widetilde \bfx^n) - \bfA(\tildexs^n) \big) \ex^n 
        \\
	&\leq \tfrac{1}{24} \|\ev^n\|_{\bfK(\xs^n)}^2 
	+ c \|\ex^n\|_{\bfK(\xs^n)}^2
	+ c \sum_{j=1}^{p}  \|\ex^{n-j}\|_{\bfK(\xs^n)}^2 , \\
	& (\ev^n)^T \bfA(\tildexs^n) \ex^n \leq \tfrac{1}{24} \|\ev^n\|_{\bfK(\xs^n)}^2 + c \|\ex^n\|_{\bfK(\xs^n)}^2 .
\end{align*}

\noindent (v) Similarly as  in (i) we rewrite 
\begin{align*}
	 (\ev^n)^T \big(\bfg(\widetilde \bfx^n,t_n) - \bfg(\tildexs^n,t_n)\big)
	 =&\ \int_{\Ga_h^{1,n}} g^n \nu_{\Gamma_h^{1,n}}\cdot e_v^{1,n} 
	 - \int_{\Ga_h^{0,n}} g^n \nu_{\Gamma_h^{0,n}}\cdot e_v^{0,n} \\
	 =&\ \int_0^1 \frac{\d}{\d \theta} \int_{\Gamma_h^{n,\theta}} g^n \nu_{\Gamma_h^{n,\theta}}\cdot e_v^{n,\theta} \d \theta .
\end{align*}
We use the Leibniz formula and $\mat_\theta e_v^{0,n} = 0$ just as in (iii) of the proof of \cite[Proposition~5.1]{soldriven}, to finally obtain
\begin{align*}
	(\ev^n)^T \big(\bfg(\widetilde \bfx^n,t_n) - \bfg(\tildexs^n,t_n)\big) 
	\leq &\  c \| e_v^n \|_{L^2(\Gamma_h[\tildexs^n])} \, 
	\| \tildeexth^{n} \|_{H^1(\Gamma_h[\tildexs^n])} \\
	\leq &\ c \| \ev^n \|_{\bfK(\tildexs^n)} \|\tildeex^n\|_{\bfK(\tildexs^n)}^2 \\
	\leq &\  \tfrac{1}{24} \|\ev^n\|_{\bfK(\xs^n)}^2 
	+ c \sum_{j=1}^{p}  \|\ex^{n-j}\|_{\bfK(\xs^n)}^2 .
\end{align*}

\noindent (vi) The term with the defect is estimated as
\begin{align*}
	(\ev^n)^T \bfM(\xs^n) \dv^n 
	= &\ (\ev^n)^T \bfK(\xs^n)^{1/2} \bfK(\xs^n)^{-1/2} \bfM(\xs^n) \dv^n \\
	\leq &\ \|\ev^n\|_{\bfK(\xs^n)} \|\dv^n\|_{\star,\xs^n} 
	\leq \tfrac{1}{24} \|\ev^n\|_{\bfK(\xs^n)}^2 + c \|\dv^n\|_{\star,\xs^n}^2 .
\end{align*}

\noindent Finally, by combining all these estimates, using multiple absorptions, with sufficiently small $\vartheta$  we finally obtain
\begin{equation}
\label{eq: final estimate for error in v}
	\|\ev^n\|_{\bfK(\xs^n)}^2 \leq c \|\ex^n\|_{\bfK(\xs^n)}^2 + c \sum_{j=1}^{p}  \|\ex^{n-j}\|_{\bfK(\xs^n)}^2 + c \|\dv^n\|_{\star,\xs^n}^2 .
\end{equation}

(b) \textit{Estimates for ODE.} 
We rewrite the second equation of \eqref{eq: error equations for BDF} as 
$$
\frac{1}{\tau} \sum_{j=p}^n \delta_{n-j} \ex^{j}  =\ \ev^n - \widehat\dx^n,  
$$
with $\delta_j=0$ for $j>p$ and 
$$
\widehat\dx^n = \dx^n + \frac{1}{\tau} \sum_{j=0}^{p-1} \delta_{n-j} \ex^{j},
$$
where we note that $\widehat\dx^n = \dx^n$ for $n\ge2p$.  With the coefficients of the power series
$$
\mu(\zeta) = \sum_{n=0}^\infty \mu_n \zeta^n = \frac 1{\delta(\zeta)}
$$
we then have, for $n\ge p$,
$$
\ex^n = \tau \sum_{j=p}^n \mu_{n-j}  (\ev^j-\widehat\dx^j).
$$
By the zero-stability of the BDF method of order $p\le 6$ (which states that all zeros of $\delta(\zeta)$ are outside the unit circle with the exception of the simple zero at $\zeta=1$), the coefficients $\mu_n$ are bounded: $|\mu_n| \le c$ for all $n$.

Taking the $K(\xs^n)$ norm on both sides and recalling that by Lemma~\ref{lemma: matrix derivatives} all these norms are uniformly equivalent for $0\le n\tau\le T$, we obtain with the Cauchy--Schwarz inequality
\begin{align*}
\| \ex^n \|_{K(\xs^n)}^2 &\le c \tau \sum_{j=p}^n \| \ev^j-\widehat\dx^j \|_{K(\xs^j)}^2 \\
&\le c \tau \sum_{j=p}^n \| \ev^j \|_{K(\xs^j)}^2 +c \tau \sum_{j=p}^n \| \dx^j \|_{K(\xs^j)}^2 + c \sum_{i=0}^{p-1} \| \ex^i\|_{K(\xs^i)}^2.
\end{align*}
Combining this inequality with \eqref{eq: final estimate for error in v} and using a discrete Gronwall inequality then yields the result. \qed
\end{proof}

\section{Consistency error}
\label{section: consistency}

In this section we show that the consistency errors, that is, the defects defined by \eqref{eq: BDF for exact sol} and obtained by inserting the interpolated exact solution into the numerical method,  are bounded in the required norms by $ C(h^k+\tau^p)$ for the finite element method of polynomial degree $k$ and the $p$-step BDF method. 
%
%
%
%
%
\newcommand{\Ih}{\widetilde{I}_h}

\bbk
Let us first recall the formula for the defect of the spatial semi-discretization $d_{h,v}(\cdot,t)$ from Section~8 of \cite{soldriven}, for $\psi_h \in S_h[\xs(t)]^3$:
\begin{align*}
	&\int_{\Ga_h[\xs\t]} \!\!\! d_{h,v}(\cdot,t) \cdot \psi_h =\ \int_{\Ga_h[\xs\t]}\!\! \Ih v(\cdot,t) \cdot \psi_h
	+ \alpha \int_{\Ga_h[\xs\t]} \!\!\! \nbgh \Ih v(\cdot,t) \cdot \nbgh \psi_h \\
	&\hskip 2cm + \beta \int_{\Ga_h[\xs\t]} \!\!\! \nbgh \Ih X(\cdot,t) \cdot \nbgh \psi_h -  \int_{\Ga_h[\xs\t]}  \!\!\! g(\cdot,t) \,\nu_{\Ga_h[\xs\t]} \cdot \psi_h ,
\end{align*}
which satisfies the following bounds.
\begin{lemma}{\rm \cite[Lemma~8.1]{soldriven}}
	Let the surface $X$ and its velocity $v$ be sufficiently smooth. Then there exists a constant $c>0$ (independent of $t$) such that for all $h\leq h_0$, with a sufficiently small $h_0>0$, and for all $t\in[0,T]$, the defects $d_{h,v}$ of the $k$th-degree finite element interpolation are bounded as
	\begin{align*}
		\|d_{h,v}(\cdot,t)\|_{H_h\inv(\Ga(\Xs))} \leq c h^k .
	\end{align*}
\end{lemma}
We will now bound the defect of the full discretization.
\ebk

\begin{lemma}
\label{lemma: defect estimates for BDF}
	Let the surface $X$ and its velocity $v$ be sufficiently smooth. Then there exist $h_0>0$ and $\tau_0>0$  such that for  all $h\leq h_0$  and for all $\tau\le\tau_0$,
	the  consistency errors 
	are bounded as
	\begin{align*}
	&\|\dv^n\|_{\star,\xs^n}=\|d_v^n\|_{H_h\inv(\Ga(\Xs(t_n)))} \, \leq c \big( \tau^p + h^k \big) , \\
	&\|\dx^n\|_{\bfK(\xs^n)}= \|d_x^n\|_{H^1(\Ga(\Xs(t_n)))} \leq c \tau^p ,
	\end{align*}
	where  $c$ is independent of $h$, $\tau$ and $n$ with   $n\tau \leq T$.
\end{lemma}

\begin{proof}
For the defect in $v$, the corresponding finite element function $d_v^n \in S_h[\tildexs^n]$ with nodal values $\dv^n$ satisfies the following: for all finite element functions $\bar\psi_h \in S_h[\xs^n]$ and the corresponding $\psi_h\in S_h[\tildexs^n]$ with the same nodal values,
\begin{equation}
\label{eq: defect def for v}
	\begin{aligned}
		\int_{\Ga_h[\xs^n]} \!\!\!\! d_v^n \cdot \bar\psi_h 
		=&\ \int_{\Ga_h[\tildexs^{n}]} \!\!\!\! \Ih v(\cdot,t_n) \cdot \psi_h  
		+ \alpha \int_{\Ga_h[\tildexs^{n}]} \!\!\!\! \nbgh \Ih v(\cdot,t_n) \cdot \nbgh \psi_h \\
		&\ + \beta \int_{\Ga_h[\tildexs^{n}]} \!\!\!\! \nbgh \Ih X(\cdot,t_n) \cdot \nbgh \psi_h
		- \int_{\Ga_h[\tildexs^{n}]} \!\!\!\! g(\cdot,t_n) \nu_{\Ga_h[\tildexs^{n}]} \cdot \psi_h ,
	\end{aligned}
\end{equation}
where $\Ih v(\cdot,t_n), \Ih X(\cdot,t_n)\in S_h[\tildexs^{n}]^3$ denote the finite element interpolation of $v(\cdot,t_n)$ and $X(\cdot,t_n)$, respectively, on $\Ga_h[\tildexs^{n}]$.
Let us first rewrite \eqref{eq: defect def for v}, by subtracting the weak form of the problem \eqref{weak form}.  For the first term on the right-hand side, by adding and subtracting, this yields
\begin{align*}
	&\int_{\Ga_h[\tildexs^{n}]} \!\!\!\!\!\! \Ih v(\cdot,t_n) \cdot \psi_h  - \int_{\Ga(X(t_n))} \!\!\!\!\!\!\!\! v(\cdot,t_n) \cdot \psi_h^l \\
	&=\ \int_{\Ga_h[\tildexs^{n}]} \!\!\!\!\!\! \Ih v(\cdot,t_n) \cdot \psi_h - \int_{\Ga_h[\xs^n]} \!\!\!\!\!\! \Ih v(\cdot,t_n) \cdot \psi_h \\
	&\ \ \ + \int_{\Ga_h[\xs^n]} \!\!\!\!\!\! \Ih v(\cdot,t_n) \cdot \psi_h - \int_{\Ga(X(t_n))} \!\!\!\!\!\!\!\! v(\cdot,t_n) \cdot \psi_h^l .
\end{align*}
Note that the last pair is simply a spatial defect, therefore repeating the same process for all four terms, and using the spatial defect $d_{h,v}$ from Section~8 of \cite{soldriven}, we obtain 
\begin{align*}
	&\int_{\Ga_h[\xs^n]} \!\!\!\!  d_v^n \cdot \psi_h = \int_{\Ga_h[\tildexs^{n}]} \!\!\!\! \Ih v(\cdot,t_n) \cdot \psi_h - \int_{\Ga_h[\xs^n]} \!\!\!\! \Ih v(\cdot,t_n) \cdot \psi_h \\
	&\ + \alpha \int_{\Ga_h[\tildexs^{n}]} \!\!\!\! \nbgh \Ih v(\cdot,t_n) \cdot \nbgh \psi_h 
	- \alpha \int_{\Ga_h[\xs^n]} \!\!\!\! \nbgh \Ih v(\cdot,t_n) \cdot \nbgh \psi_h \\
	&\ + \beta \int_{\Ga_h[\tildexs^{n}]} \!\!\!\! \nbgh \Ih X(\cdot,t_n) \cdot \nbgh \psi_h  - \beta \int_{\Ga_h[\xs^n]} \!\!\!\! \nbgh \Ih X(\cdot,t_n) \cdot \nbgh \psi_h\\
	&\ - \int_{\Ga_h[\tildexs^{n}]} \!\!\!\! g(\cdot,t_n) \nu_{\Ga_h[\tildexs^{n}]} \cdot \psi_h + \int_{\Ga_h[\xs^n]} \!\!\!\! g(\cdot,t_n) \nu_{\Ga_h[\tildexs^{n}]} \cdot \psi_h \\
	&\ + \int_{\Ga_h[\xs^n]} \!\!\!\! d_{h,v}(\cdot,t_n) \cdot \psi_h.
\end{align*}
We estimate the defect $d_v^n$ pairwise, using similar tools as in part (a) of the proof of Proposition~\ref{propostion: stability - regularised velocity law} and recalling \eqref{dual-norm-h}.

For the first pair, we use the setting of Lemma~\ref{lemma: matrix identity for extrapolation}, and then a Cauchy--Schwarz inequality and an $L^2 - L^2 - L^\infty$ estimate yield
\begin{align*}
	&\hspace{-3mm} \Big| \int_{\Ga_h[\tildexs^{n}]} \!\!\!\! \Ih v(\cdot,t_n) \cdot \psi_h - \int_{\Ga_h[\xs^n]} \!\!\!\! \Ih v(\cdot,t_n) \cdot \psi_h \Big| \\
	=&\ \Big| \int_0^1 \int_{\Gamma_h^{n,\theta}} \psi_h^\theta (\nabla_{\Gamma_h^{n,\theta}} \cdot {\widetilde e}_h^{n,\theta} ) v_{\ast,h}^{n,\theta} \d \theta \Big| \\
	\leq &\ \int_0^1 \|\psi_h^\theta\|_{L^2(\Gamma_h^{n,\theta})}
	\|\nabla_{\Gamma_h^{n,\theta}} \cdot {\widetilde e}_h^{n,\theta} \|_{L^2(\Gamma_h^{n,\theta})} 
	\|v_{\ast,h}^{n,\theta}\|_{L^\infty(\Gamma_h^{n,\theta})} \d \theta \\
	\leq &\ c  \|\psi_h^0\|_{L^2(\Ga_h^{0,n})} 
	\| {\widetilde e}_h^{n,0} \|_{H^1(\Ga_h^{0,n})} 
	\| v_{\ast,h}^{n,0}\|_{L^{\infty}(\Ga_h^{0,n})} \\
	\leq &\ c  \|\psi_h\|_{L^2(\Ga_h[\xs^n])} 
	\| {\widetilde e}_h^{n} \|_{H^1(\Ga_h[\xs^n])} \\
	&\ \cdot 
	\Big(\| v_{\ast}(\cdot,t_n)\|_{L^{\infty}(\Ga_h[\xs^n])} + \| v_{\ast,h}(\cdot,t_n)-v_{\ast}(\cdot,t_n)\|_{L^{\infty}(\Ga_h[\xs^n])}\Big) \\
	\leq &\ c  \|\psi_h\|_{L^2(\Ga_h[\xs^n])} 
	\| {\widetilde e}_h^{n} \|_{H^1(\Ga_h[\xs^n])} (1+ch^2)
	\| v_{\ast}(\cdot,t_n)\|_{W^{1,\infty}(\Ga_h[\xs^n])}  \\
	\leq &\ c \|\tildexs^n-\xs^n\|_{\bfK(\xs^n)} \|\psi_h\|_{L^2(\Ga_h[\xs^n])}  \\ 
	\leq &\ c \tau^p \|\psi_h\|_{L^2(\Ga_h[\xs^n])}  ,
\end{align*}
where we used a $W^{1,\infty}$ interpolation estimate from \cite[Proposition~2.7]{Demlow2009}, and the last inequality follows from \eqref{peano}.

%
%

The other three pairs are again estimated similarly as above, and we finally obtain the bounds
\begin{align*}
	 \Big| \int_{\Ga_h[\tildexs^{n}]} \!\!\!\! \nbgh \Ih v(\cdot,t_n) \cdot \nbgh \psi_h 
	- &\ \int_{\Ga_h[\xs^n]} \!\!\!\! \nbgh \Ih v(\cdot,t_n) \cdot \nbgh \psi_h \Big| \\
	\leq &\ c \tau^p \|\psi_h\|_{H^1(\Ga_h[\xs^n])}\\[2mm]
	 \Big| \int_{\Ga_h[\tildexs^{n}]} \!\!\!\! \nbgh \Ih X(\cdot,t_n) \cdot \nbgh \psi_h  - &\ \int_{\Ga_h[\xs^n]} \!\!\!\! \nbgh \Ih X(\cdot,t_n) \cdot \nbgh \psi_h \Big| \\
	\leq &\ c \tau^p \|\psi_h\|_{H^1(\Ga_h[\xs^n])}\\[2mm]
	\Big| \int_{\Ga_h[\tildexs^{n}]} \!\!\!\! g(\cdot,t_n) \nu_{\Ga_h[\tildexs^{n}]} \cdot \psi_h - &\ \int_{\Ga_h[\xs^n]} \!\!\!\! g(\cdot,t_n) \nu_{\Ga_h[\tildexs^{n}]} \cdot \psi_h \Big| \\
	\leq &\ c \tau^p \|\psi_h\|_{H^1(\Ga_h[\xs^n])} .
\end{align*}
Furthermore, as shown in Lemma~8.1 of \cite{soldriven}, the spatial defect $d_{h,v}(\cdot,t_n)$ is bounded by
\begin{equation*}
	\int_{\Ga_h[\xs^n]} \!\!\!\! d_{h,v}(\cdot,t_n) \cdot \psi_h  \leq c h^k \|\psi_h\|_{H^1(\Ga_h[\xs^n])} .
\end{equation*}
Combining the above estimates, we obtain the bound $\|\dv^n\|_{\star,\xs^n} \leq c \big( \tau^p + h^k \big)$.
The defect in $X$ is  given by
\begin{equation*}
	\dx^n = \frac{1}{\tau} \sum_{j=0}^p \delta_j \xs(t_{n-j}) - \dotxs(t_n) 
\end{equation*}
and is solely  due to temporal discretization.
The bound $\|\dx^n\|_{\bfK(\xs^n)} \leq c \tau^p $ then follows by Taylor expansion.
\qed
\end{proof}


\section{Proof of Theorem~\ref{theorem: main}}
\label{section: proof completed}

The errors are decomposed using interpolations and the definition of lifts from Section~\ref{section:lifts}. We denote by $\widehat I_h v\in S_h[\xs]$ the finite element interpolation of $v$ on the interpolated surface $\Gamma_h[\xs]$ and by $I_hv =(\widehat I_h v)^l$ its lift to the exact surface $\Gamma(X)$. We write
\begin{align*}
	(v_h^n)^{L}  - v(\cdot,t_n)  =&\ \big(\widehat v_h^n  - \widehat I_h v(\cdot,t_n) \big)^{l} + \big(I_h v(\cdot,t_n)  - v(\cdot,t_n) \big) , \\
	(X_h^n)^L  - X(\cdot,t_n)  =&\ \big( \widehat X_h^n  - \widehat I_h X(\cdot,t_n)  \big)^{l} +  \big(I_h X(\cdot,t_n)  - X(\cdot,t_n) \big).
\end{align*}
The last terms  in these formulas can be bounded in the $H^1(\Gamma)$ norm by $Ch^k$, using the interpolation bounds of \cite{highorder}.

To bound the first terms on the right-hand sides, we first use the defect bounds of Lemma~\ref{lemma: defect estimates for BDF}, which then, under the mild stepsize restriction, together with the stability estimate of Proposition~\ref{propostion: stability - regularised velocity law} proves the result,
since by the norm equivalences from Lemma~\ref{lemma: technicals} and equations \eqref{M-L2}--\eqref{A-H1} we have
\begin{align*}
	\| \big(\widehat v_h^n - \widehat I_h v(\cdot,t_n)\big)^{l} \|_{L^2(\Ga(\cdot,t_n))} \leq &\
	c \| \widehat v_h^n - \widehat I_h v(\cdot,t_n) \|_{L^2(\Ga_h[\xs^n])} \\
	= &\ c \|\ev^n\|_{\bfM(\xs^n)} ,
	\\
	\| \nabla_{\Gamma} \big(\widehat v_h^n - \widehat I_h v(\cdot,t_n)\big)^{l} \|_{L^2(\Ga_h[\xs^n]))} \leq &\ 
	c \| \nabla_{\Gamma_h^*} \big(\widehat v_h^n - \widehat I_h v(\cdot,t_n) \big) \|_{L^2(\Ga_h[\xs^n])} \\
	= &\ c \|\ev^n\|_{\bfA(\xs^n)},
\end{align*}
and similarly for $\widehat X_h^n - \widehat I_h X(\cdot,t_n)$.

%
%
%
\section{A dynamic velocity law}
\label{section: dynamic}

\subsection{Weak formulation and ESFEM / BDF full discretization}
We now consider the dynamic velocity law \eqref{v-dyn}, viz.,
\begin{equation*}
    \mat v + v \nb_{\Ga(X)} \cdot v  - \alpha \laplace_{\Ga(X)} v = g(\cdot,t) \, \nu_{\Ga(X)},
\end{equation*}
where again $g:\R^3\times\R\to\R$ is a given smooth function of $(x,t)$, and $\alpha>0$ is a fixed parameter. This problem is considered together with the ordinary differential equation \eqref{velocity} for the positions $X$ determining the surface $\Gamma(X)$. Initial values are specified for~$X$ and $v$.

The weak formulation of the dynamic velocity law \eqref{v-dyn} reads as follows:
Find $v(\cdot,t) \in  W^{1,\infty}(\Ga(X(\cdot,t)))^3 $ such that for all test functions $\psi(\cdot,t)  \in H^1(\Ga(X(\cdot,t) ))^3$ with vanishing material derivative,
\begin{equation}
\label{weak form - dynamic}
    \begin{aligned}
      \frac{\d}{\d t}  \int_{\Ga(X)} \!\!v \cdot \psi 
        &\ + \alpha \int_{\Ga(X)} \!\! \nabla_{\Gamma(X)} v \cdot \nabla_{\Gamma(X)} \psi 
        = \int_{\Ga(X)} \!\! g \,\nu_{\Ga(X)} \cdot \psi ,
    \end{aligned}
\end{equation}
together with the ordinary differential equation \eqref{velocity} for the positions $X$ determining the surface $\Gamma(X)$. The finite element space discretization is done in the usual way. We forego the straightforward formulation and
immediately present the matrix--vector formulation of the semi-discretization. As in Section~\ref{subsection:DAE}, the nodal vectors $\bfv(t)\in\R^{3N}$ of the finite element function $v_h(\cdot,t)$, together with the surface nodal vector $\bfx(t)\in\R^{3N}$ satisfy a system of ordinary differential equations with matrices and driving term as in Section~\ref{subsection:DAE}:
\begin{equation}
\label{eq: DAE form - dynamic}
    \begin{aligned}
        \diff \Big(\bfM(\bfx) \bfv\Big) + \bfA(\bfx)\bfv =&\ \bfg(\bfx,t) , \\
        \dot\bfx =&\ \bfv .
    \end{aligned}
\end{equation}
We apply a $p$-step linearly implicit BDF method to the above ODE system with a step size $\tau >0 $: with $t_n= n \tau \leq T $ and with the extrapolated nodal vector $\tildexs^n$ defined by \eqref{eq: extrapolation of u def}, 
the new nodal vectors of velocity and position, $\bfv^n$ and $\bfx^n$, respectively, are determined from the following system of linear equations:
\begin{equation}
\label{eq: BDF for dynamic}
	\begin{aligned}
		\frac{1}{\tau} \sum_{j=0}^p &\,\delta_j \bfM(\tildex^{n-j})\bfv^{n-j} + \ \bfA(\tildex^n)\bfv^n =\bfg(\tildex^n,t)\\
		\frac{1}{\tau} \sum_{j=0}^p &\,\delta_j \bfx^{n-j} =\ \bfv^n .
	\end{aligned}
\end{equation}
As in Section~\ref{section: problem}, the nodal vector $\bfx^n$ defines the discrete surface $\Gamma_h[\bfx^n]=\Gamma(X_h^n)$, which is to approximate the exact surface $\Gamma(X)$, and we obtain the position and velocity approximations \eqref{x-v-approx}.

\subsection{Statement of the error bound}
The following result is the analogue of Theorem~\ref{theorem: main} for the dynamic velocity law. We use the same notation for the lifted approximations.

\begin{theorem}
\label{theorem: main-dyn}
    Consider the ESFEM / BDF linearly implicit full discretization \eqref{eq: BDF for dynamic} of the dynamic velocity equation \eqref{v-dyn}, using finite elements of polynomial degree~$k\ge 2$ and BDF methods of order $p\le 5$.
    We assume quasi-uniform admissible triangulations of the initial surface and initial values chosen by finite element interpolation of the initial data for $X$.
    Suppose that the problem admits an exact solution $X,v$ that is sufficiently smooth (say, of class  $C([0,T],H^{k+1})\cap C^{p+1}([0,T],W^{1,\infty})$) on the time interval $0\le t \le T$,  and that the flow map $X(\cdot,t):\Gamma_0\to \Gamma(t)\subset\R^3$ is non-degenerate for $0\le t \le T$, so that $\Gamma(t)$ is a regular surface. Suppose further that the starting values are sufficiently accurate: for $i=0,\dots, p-1$,
  $$
  \| (X_h^i)^L - X(\cdot,i\tau) \|_{H^1(\Ga^0)^3} +  \| (v_h^i)^L - v(\cdot,i\tau) \|_{H^1(\Ga^0)^3}\le C_0  (h^k+\tau^{p}).
  $$
    Then, there exist $h_0 >0$, $\tau_0>0$ and $c_0>0$ such that for all mesh widths $h \leq h_0$ and step sizes $\tau\le\tau_0$ satisfying the mild stepsize restriction
    $
    \tau^p \le c_0 h,
    $
    the following error bounds hold over the exact surface $\Ga(t_n)=\Ga(X(\cdot,t_n))$
uniformly  for $0\le t_n=n\tau \le T$:
    \begin{equation*}
 	\begin{aligned}
   		\|(x_h^n)^{L} - \mathrm{id}_{\Gamma(t_n)}\|_{H^1(\Ga(t_n))^3} &\leq C(h^k+\tau^p),\\
		 \|(v_h^n)^{L} - v(\cdot,t_n)\|_{L^2(\Ga(t_n))^3} +   \biggl( \sum_{j=p}^n \| (v_h^j)^{L} - v(\cdot,t_j)&\|_{H^1(\Ga(t_j))^3}^2 \biggr)^{1/2}\\
	 &\leq C(h^k+\tau^p).
 	\end{aligned}
    \end{equation*}
    The constant $C$ is independent of  $h$ and $\tau$ and $n$ with $n\tau\le T$, but depends on bounds of higher derivatives of the solution $(X,v)$,  and on the length $T$ of the time interval.
\end{theorem}

\subsection{Auxiliary results by Dahlquist and Nevanlinna \& Odeh}

While the formulations of Theorems~\ref{theorem: main} and ~\ref{theorem: main-dyn} are very similar, the proofs differ substantially in the stability analysis. In this subsection
 we recall two important results that combined permit us to use energy estimates for BDF methods up to order 5: the first result is from Dahlquist's $G$-stability theory, and the second one from the multiplier technique of Nevanlinna and Odeh.
 These results have previously been used in the error analysis of BDF methods for various parabolic problems in
\cite{AkrivisLubich_quasilinBDF,AkrivisLiLubich_quasilinBDF,KovacsPower_quasilinear,LubichMansourVenkataraman_bdsurf}.

\begin{lemma}[Dahlquist \cite{Dahlquist}]
\label{lemma: Dahlquist}
    Let $\delta(\zeta)=\sum_{j=1}^p\delta_j\zeta^j$ and $\mu(\zeta)=\sum_{j=1}^p\mu_j\zeta^j$ be polynomials of degree at most $p$ (at least one of them of degree $p$) that have no common divisor. Let $\la \ \cdot , \cdot \ \ra$ denote an inner product on $\R^N$. 
    If
    \begin{equation*}
        \textnormal{Re} \frac{\delta(\zeta)}{\mu(\zeta)} > 0, \qquad \textrm{for} \quad |\zeta|<1,
    \end{equation*}
    then there exists a symmetric positive definite matrix $G = (g_{ij}) \in \R^{p\times p}$ 
    such that for all $\bfw_0,\dotsc,\bfw_p\in\R^N$
    \begin{equation*}
        \Big\la \sum_{i=0}^p \delta_i \bfw_{p-i} , \sum_{i=0}^p \mu_i \bfw_{p-i}  \Big\ra \ge \sum_{i,j=1}^p g_{ij} \la \bfw_i , \bfw_j \ra - \sum_{i,j=1}^p g_{ij} \la \bfw_{i-1} , \bfw_{j-1} \ra .
      \end{equation*}
\end{lemma}

In view of the following result, the choice $\mu(\zeta)=1-\eta\zeta$ together with the polynomial $\delta(\zeta)$ of the BDF methods will play an important role later on.
\begin{lemma}[Nevanlinna \& Odeh \cite{NevanlinnaOdeh}]
\label{lemma: NevanlinnaOdeh multiplier}
    If $p\leq5$, then there exists $0\leq\eta<1$ \st\ for $\delta(\zeta)=\sum_{\ell=1}^p \frac{1}{\ell}(1-\zeta)^\ell$,
    \begin{equation*}
        \textnormal{Re} \,\frac{\delta(\zeta)}{1-\eta\zeta} > 0, \qquad \textrm{for} \quad |\zeta|<1.
    \end{equation*}
    The smallest possible values of $\eta$ are found to be 
    $\eta= 0,  0,  0.0836, 0.2878,  0.8160$ 
    for $p=1,\dotsc,5$, respectively.
\end{lemma}

\subsection{Error equations}
By using the same notations as in the previous sections for the nodal vectors of the exact positions $\xs^n\in\R^{3N}$ and of the  exact velocity $\vs^n\in\R^{3N}$, and for their defects $\dv^n$ and $\dx^n$, 
we obtain that they fulfil the following equations:
\begin{equation*}
	\begin{aligned}
		\frac{1}{\tau} \sum_{j=0}^p \delta_j \bfM(\tildexs^{n-j})\vs^{n-j} + \ \bfA(\tildexs^n)\vs^n &=\bfg(\tildex^n,t)+ \bfM(\xs^n) \dv^n,\\
		 \frac{1}{\tau} \sum_{j=0}^p \delta_j \xs^{n-j} &=\vs^n + \dx^n .
	\end{aligned}
%
\end{equation*}
By subtracting the above equations from \eqref{eq: BDF for dynamic}, we obtain the error equations for the surface nodes and velocity:
\begin{equation}
\label{eq:exact solution for dyn velocity law}
    \begin{aligned}
        &\ \hspace{-.5cm} \bfM(\xs^n) \frac{1}{\tau} \sum_{j=0}^p \delta_j \ev^{n-j} 
        + \bfA(\xs^n)\ev^n \\
        =&\ 
           -  \frac{1}{\tau}\sum_{j=1}^p \delta_j \big( \bfM(\xs^{n-j}) - \bfM(\xs^{n})\bigr) \ev^{n-j}
           - \frac{1}{\tau} \sum_{j=0}^p \delta_j \big( \bfM(\tildexs^{n-j}) - \bfM(\xs^{n-j})\bigr) \ev^{n-j}\\
          &\ - \frac{1}{\tau} \sum_{j=0}^p \delta_j \big( \bfM(\tildex^{n-j}) - \bfM(\tildexs^{n-j})\bigr)(\vs^{n-j} + \ev^{n-j})\\
        &\ - \big(\bfA(\tildexs^n) - \bfA(\xs^n)\big) \ev^n - \big(\bfA(\widetilde\bfx^n) - \bfA(\tildexs^n)\big) (\vs^n+\ev^n) \\[1mm]
        &\ + \bfg(\tildex^n,t_n) - \bfg(\tildexs^n,t_n) - \bfM(\xs^n) \dv^n \\[2mm]
        &\ \hspace{-.5cm} \frac{1}{\tau} \sum_{j=0}^p \delta_j \ex^{n-j} = \ev^n - \dx^n  .
    \end{aligned}
\end{equation}

\subsection{Stability}

We then have the following stability result.

\begin{proposition}
\label{proposition: stability - dynamic velocity law}
    Under the smallness assumptions of Proposition~\ref{propostion: stability - regularised velocity law} for the defects and the errors in the initial values, 
    the following error bound holds for BDF methods of order $p\le 5$ for $n \tau \leq T$:
    \begin{equation}
    \label{eq: error estimate - dynamic}
	    \begin{aligned}
		    & \|\ex^{n}\|_{\bfK(\xs^{n})}^2 
		    + \|\ev^n\|_{\bfM(\xs^n)}^2 
		    + \tau \sum_{j=p}^n \|\ev^j\|_{\bfA(\xs^j)}^2 
		    \\
		    &\leq  C \tau \sum_{j=p}^n \Big( \|\dx^j\|_{\bfK(\xs^j)}^2 + \|\dv^j\|_{\star,\xs^j}^2 \Big) + c \|\dv^n\|_{\star,\xs^n}^2 
		    \\
		    &\quad + C\sum_{i=0}^{p-1}\Big( \|\ex^i\|_{\bfK(\xs^i)}^2 + \|\ev^i\|_{\bfM(\xs^i)}^2 \Big).		    
	    \end{aligned}
    \end{equation}
    The constant $C$ is independent of $h$, $\tau$ and $n$, but depends on  $T$.
\end{proposition}
\begin{proof} 
We test the first error equation in \eqref{eq:exact solution for dyn velocity law} with $\ev^n-\eta \ev^{n-1}$ to obtain 
$$
(\ev^n-\eta \ev^{n-1})^T \bfM(\xs^n) \frac{1}{\tau} \sum_{j=0}^p \delta_j \ev^{n-j} + (\ev^n-\eta \ev^{n-1})^T\bfA(\xs^n)\ev^n =\rho^n,
$$
where the right-hand term $\rho^n$ can be estimated by the same arguments as in part (a) of the proof of Proposition~\ref{propostion: stability - regularised velocity law}. On the left-hand side we have a term containing the stiffness matrix $\bfA(\xs^n)$, which is estimated from below as follows using Lemmas~\ref{lemma: matrix derivatives} and~\ref{lemma: matrix identity for extrapolation}:
\begin{align*}
(\ev^n-&\eta \ev^{n-1})^T\bfA(\xs^n)\ev^n \ge  \| \ev^n \|_{\bfA(\xs^n)}^2 - \eta \| \ev^{n-1} \|_{\bfA(\xs^n)} \| \ev^n \|_{\bfA(\xs^n)}
\\
&\ge  \| \ev^n \|_{\bfA(\xs^n)}^2 - \eta (1+c\tau)\| \ev^{n-1} \|_{\bfA(\xs^{n-1})} \| \ev^n \|_{\bfA(\xs^n)}
\\
&\ge (1-\tfrac12\eta-c\tau) \| \ev^n \|_{\bfA(\xs^n)}^2 - (\tfrac12\eta+c\tau) \| \ev^{n-1} \|_{\bfA(\xs^{n-1})}^2.
\end{align*}
The other term on the left-hand side, which contains the mass matrix $\bfM(\xs^n)$, is estimated from below using 
Lemmas~\ref{lemma: Dahlquist} and~\ref{lemma: NevanlinnaOdeh multiplier}. Let us introduce
\begin{equation*}
	\bfE_{\bfv}^n = \big(\ev^{n-p+1}, \dotsc, \ev^{n-1}, \ev^n \big) 
\end{equation*}
and the norm 
\begin{equation*}
	|\bfE_{\bfv}^n|_{G,\xs^n}^2 = \sum_{i,j=1}^p g_{ij} (\ev^{n-p+i})^T \bfM(\xs^n) \ev^{n-p+j} ,
\end{equation*}
which satisfies the norm equivalence relation
\begin{equation}
\label{eq: G norm equivalence}
	\lambda_{\min} \sum_{i=1}^p \|\ev^{n-p+i}\|_{\bfM(\xs^n)}^2 
	\leq |\bfE_{\bfx}^n|_{G,\xs^n}^2
	\leq \lambda_{\max} \sum_{i=1}^p \|\ev^{n-p+i}\|_{\bfM(\xs^n)}^2 ,
\end{equation}
where $\lambda_{\min}$ and $\lambda_{\max}$ are the smallest and largest eigenvalue 
of the symmetric positive definite matrix $G=(g_{ij})$ of Lemma~\ref{lemma: Dahlquist}.
Hence we obtain from Lemmas~\ref{lemma: Dahlquist} and~\ref{lemma: NevanlinnaOdeh multiplier}
$$
(\ev^n-\eta \ev^{n-1})^T \bfM(\xs^n)  \sum_{j=0}^p \delta_j \ev^{n-j} \ge
|\bfE_{\bfv}^n|_{G,\xs^n}^2 - |\bfE_{\bfv}^{n-1}|_{G,\xs^{n}}^2\,,
$$
where we note that by Lemma~\ref{lemma: matrix derivatives},
$$
|\bfE_{\bfv}^{n-1}|_{G,\xs^{n}}^2 \le (1+c\tau)|\bfE_{\bfv}^{n-1}|_{G,\xs^{n-1}}^2,
$$
so that altogether we have
\begin{align*}
&|\bfE_{\bfv}^n|_{G,\xs^n}^2 - (1+c\tau)|\bfE_{\bfv}^{n-1}|_{G,\xs^{n-1}}^2 \\
&+  \tau(1-\tfrac12\eta-c\tau) \| \ev^n \|_{\bfA(\xs^n)}^2 - \tau (\tfrac12\eta+c\tau) \| \ev^{n-1} \|_{\bfA(\xs^{n-1})}^2 \le \tau\rho^n.
\end{align*}
Using these inequalities from 1 to $n$ yields for sufficiently small $\tau$, with a positive constant $\gamma$,
$$
|\bfE_{\bfv}^n|_{G,\xs^n}^2  + \gamma\tau \sum_{j=0}^n e^{c(n-j)\tau}  \| \ev^j \|_{\bfA(\xs^j)}^2
\le e^{cn\tau} |\bfE_{\bfv}^{0}|_{G,\xs^{0}}^2 + \tau  \sum_{j=0}^n e^{c(n-j)\tau} \rho^j. 
$$
Using this bound together with estimates for $\rho^j$ and $\ex^j$ obtained in the same way as in the proof of Proposition~\ref{propostion: stability - regularised velocity law} then yields the stated result.
\qed
\end{proof}

Together with bounds for the consistency errors $\dv^n$ and $\dx^n$, which are proven in the same way as in Section~\ref{section: consistency}, the stability bounds of Proposition~\ref{proposition: stability - dynamic velocity law} then yield the $O(h^k+\tau^p)$ error bounds of Theorem~\ref{theorem: main-dyn}.

\section{Coupling with diffusion on the surface}
\label{section: coupled}

Let us now turn to the parabolic surface PDE coupled with the regularised velocity law. We consider the following coupled problem of an evolving surface driven by diffusion on the surface, for which the ESFEM semi-discretization was studied in \cite{soldriven}:
\begin{equation}
\label{eq: coupled problem}
	\begin{aligned}
		\mat u + u \nb_{\Ga(X)} \cdot v - \laplace_{\Ga(X)} u   =&\  f(u, \nabla_{\Ga(X)} u), \\
		v - \alpha \laplace_{\Ga(X)} v + \beta H_{\Ga(X)}\nu_{\Ga(X)}=&\ g(u, \nb_{\Ga(X)} u ) \nu_{\Ga(X)} \\
		\pa_t X(q,t) =&\ v(X(q,t),t),
	\end{aligned}
\end{equation}
with $\alpha>0$ and $\beta\ge 0$.
The weak formulation and the ESFEM spatial semi-discretization, also in its matrix--vector formulation, are given in Section~2 of \cite{soldriven}. The finally obtained coupled system of differential-algebraic equations for the vectors of nodal values $\bfu(t) \in \R^N$, $\bfv(t) \in \R^{3N}$, and $\bfx(t) \in \R^{3N}$ reads, with the matrices of Section~\ref{subsection:DAE}:
\begin{equation}
\label{eq: DAE form - coupled}
	\begin{aligned}
		\diff \Big(\bfM(\bfx)\bfu\Big) + \bfA(\bfx)\bfu =&\   \bff(\bfx,\bfu), \\
		\bfK(\bfx) \bfv +\beta \bfA(\bfx)\bfx=&\ \bfg(\bfx,\bfu),\\
		\dot \bfx =&\ \bfv.
	\end{aligned}
\end{equation}
The right-hand side vectors are defined slightly differently from Section~\ref{subsection:DAE}. They are given by
\begin{equation*}
	\begin{aligned}
		\bff(\bfx,\bfu)\vert_j &= \int_{\Ga_h[\bfx]} f(u_h,\nbgh u_h) \, \phi_j[\bfx],
		\\
		\bfg(\bfx,\bfu)\vert_{3(j-1)+\ell} &= \int_{\Ga_h[\bfx]} g(u_h,\nbgh u_h) \,\bigl(\nu_{\Ga_h[\bfx]}\bigr)_\ell \, \phi_j[\bfx],
	\end{aligned}
\end{equation*}
for $j = 1,  \dotsc, N,$ and $\ell=1,2,3$.
 
The linearly implicit BDF discretization then reads as follows: with the extrapolated position vectors $\widetilde \bfx^n$ defined by \eqref{eq: extrapolation of u def},
\newcommand{\tildeu}{\widetilde\bfu}
\newcommand{\tildeus}{\widetilde{\bfu}_{\ast}}
\begin{equation}
\label{eq: BDF for coupled}
	\begin{aligned}
			\frac{1}{\tau} \sum_{j=0}^p \delta_j \bfM(\tildex^{n-j})\bfu^{n-j} +  \bfA(\tildex^n)\bfu^n &= \bff(\tildex^n,\tildeu^n) , \\
		\bfK(\tildex^n) \bfv^n +\beta \bfA(\tildex^n)\bfx^n &= \bfg(\tildex^n,\tildeu^n), \\
		\frac{1}{\tau} \sum_{j=0}^p \delta_j \bfx^{n-j} &= \bfv^n .
	\end{aligned}
\end{equation}

Full discretizations using BDF methods of parabolic PDEs on an evolving surface with a {\it given} velocity have been studied in \cite{LubichMansourVenkataraman_bdsurf}. The combination of the proofs of Lemma~4.1 and Theorem~5.1 of \cite{LubichMansourVenkataraman_bdsurf} with the error analysis of the ESFEM semi-discretization in \cite{soldriven} and 
with the proof of Theorem~\ref{theorem: main} in the present paper yields the following convergence theorem. We omit the details of the proof.

\begin{theorem}
\label{theorem: main-coupled}
    Consider the ESFEM / BDF linearly implicit full discretization \eqref{eq: BDF for coupled} of the coupled surface-evolution equation \eqref{eq: coupled problem}, using finite elements of polynomial degree~$k\ge 2$ and BDF methods of order $p\le 5$.
    We assume quasi-uniform admissible triangulations of the initial surface and initial values chosen by finite element interpolation of the initial data for $X$.
    Suppose that the problem admits an exact solution $u,X,v$ that is sufficiently smooth (say, of class  $C([0,T],H^{k+1})\cap C^{p+1}([0,T],W^{1,\infty})$) on the time interval $0\le t \le T$,  and that the flow map $X(\cdot,t):\Gamma_0\to \Gamma(t)\subset\R^3$ is non-degenerate for $0\le t \le T$, so that $\Gamma(t)$ is a regular surface. Suppose further that the starting values are sufficiently accurate.
    Then, there exist $h_0 >0$, $\tau_0>0$ and $c_0>0$ such that for all mesh widths $h \leq h_0$ and step sizes $\tau\le\tau_0$ satisfying the mild stepsize restriction
    $
    \tau^p \le c_0 h,
    $
    the following error bounds hold over the exact surface $\Ga(t_n)=\Ga(X(\cdot,t_n))$
uniformly  for $0\le t_n=n\tau \le T$:
    \begin{equation*}
 	\begin{aligned}
	 \|(u_h^n)^{L} - u(\cdot,t_n)\|_{L^2(\Ga(t_n))^3} +   \biggl( \sum_{j=p}^n \| (u_h^j)^{L} - u(\cdot,t_j)&\|_{H^1(\Ga(t_j))^3}^2 \biggr)^{1/2}\\
	 &\leq C(h^k+\tau^p),\\
	 		 \|(v_h^n)^{L} - v(\cdot,t_n)\|_{L^2(\Ga(t_n))^3} +   \biggl( \sum_{j=p}^n \| (v_h^j)^{L} - v(\cdot,t_j)&\|_{H^1(\Ga(t_j))^3}^2 \biggr)^{1/2}\\
	 &\leq C(h^k+\tau^p),\\[1mm]
   		\|(x_h^n)^{L} - \mathrm{id}_{\Gamma(t_n)}\|_{H^1(\Ga(t_n))^3} &\leq C(h^k+\tau^p).\\
 	\end{aligned}
    \end{equation*}
    The constant $C$ is independent of  $h$ and $\tau$ and $n$ with $n\tau\le T$,  but depends on bounds of higher derivatives of the solution $(u,v,X)$,  and on the length $T$ of the time interval.
\end{theorem}

\section{Numerical experiments}
\label{section: numerics}

\bbk
\subsection{Forced mean curvature flow}
\ebk

We performed numerical experiments for the velocity law \eqref{v-eq}: for $x=X(q,t) \in \Gamma(t)$ with $q \in \Ga_0$,
\begin{equation}
\label{eq: numerics problem}
	\begin{aligned}
		v(x,t) - \alpha \laplace_{\Ga(t)} v(x,t) =&\ -\beta H_{\Ga(t)}(x)\, \nu_{\Ga(t)}(x) +g\bigl(x,t\bigr)\, \nu_{\Ga(t)}(x), \\
		\partial_t X(q,t) =&\ v(X(q,t),t) ,
	\end{aligned}
\end{equation}
where the inhomogeneity $g:R^3 \times [0,T] \to \R$ is chosen such that the exact solution is $X(q,t)=r(t) q$, with $q$ on the unit sphere $\Gamma_0$. The function $r$ satisfies the logistic differential equation:
\begin{align*}
	\dot r\t =&\ \Big(1 - \tfrac{r_1}{r\t}\Big) r\t, \qquad t \in [0,T], \\
	r(0) =&\ r_0 ,
\end{align*}
with $r_1 \geq r_0 =1$, i.e.\ $r\t=r_0r_1 \big( r_0(1-e^{-t}) + r_1 e^{-t}\big)\inv$. 

Therefore, the velocity is simply given by, for $x\t = X(q,t)$,
\begin{align*}
	v(x\t,t) = &\ \dot x\t 
	= \dot r \t p = \Big(1 - \tfrac{r_1}{r\t}\Big) r\t p = \Big(1 - \tfrac{r_1}{r\t}\Big) x\t .
\end{align*}

The numerical experiments were performed in Matlab, using a quadratic approximation of the initial surface $\Ga_0$ and using the quadratic ESFEM implementation from \cite{highorder}, and linearly implicit BDF methods of various orders.

Let $(\mathcal{T}_k)_{k=1,2,\dotsc,m}$ and $(\tau_k)_{k=1,2,\dotsc,n}$ be a series of quadratic initial meshes and time steps, respectively, such that $2 \tau_k = \tau_{k-1}$, with $\tau_1=0.1$, \bbk where the meshes are generated independently.\ebk

We computed the fully discrete numerical solution of the above problem, with parameters $\alpha=1$ and $\beta=1$, for each mesh and stepsize using the second order BDF method and second order ESFEM.
In Figures~\ref{fig: timeconv} and \ref{fig: spaceconv} we report on the following errors of the  quadratic ESFEM / BDF2 full discretization
\begin{equation*}
	\|(x_h^n)^{L} - \mathrm{id}_{\Gamma(t_n)}\|_{L^2(\Ga(t_n))^3} \andquad \|\nbg \big( (x_h^n)^{L} - \mathrm{id}_{\Gamma(t_n)} \big) \|_{L^2(\Ga(t_n))^3} 
\end{equation*}
at time $T=N\tau=5$. The logarithmic plots show the errors against time step size $\tau$ (in Figure~\ref{fig: timeconv}), and against the mesh width $h$ (in Figure~\ref{fig: spaceconv}).

The different lines correspond to different mesh refinements and to different time step sizes in Figure \ref{fig: timeconv} and Figure~\ref{fig: spaceconv}, respectively.
In both figures we can observe two regions: In Figure~\ref{fig: timeconv}, a region where the temporal discretization error dominates, matching to the $O(\tau^2)$ order of convergence of our theoretical result, and a region, with small stepsizes, where the space discretization error dominates (the error curves are flattening out). In Figure~\ref{fig: spaceconv}, the same description applies, but with reversed roles. First the space discretization error dominates, while for finer meshes the temporal error  dominates.
The convergence in time, see Figure~\ref{fig: timeconv}, can be nicely observed in agreement with the theoretical results (note the reference line), whereas we observe better $L^2$ norm convergence rates ($O(h^3)$) for the space discretization, see Figure~\ref{fig: spaceconv}, than shown in Theorem~\ref{theorem: main} for the $H^1$ norm (only $O(h^2)$). \bbk This phenomenon is due to the fact that in the defect estimates we use the interpolation instead of a Ritz projection (which is hard to define in this setting), therefore have a defect estimate of order two. However, the classical optimal $L^2$ norm convergence rates of $O(h^3)$ are nevertheless observed.\ebk

\begin{figure}[htbp]
	\centering
	\includegraphics[width=\textwidth,height=.42\textheight]{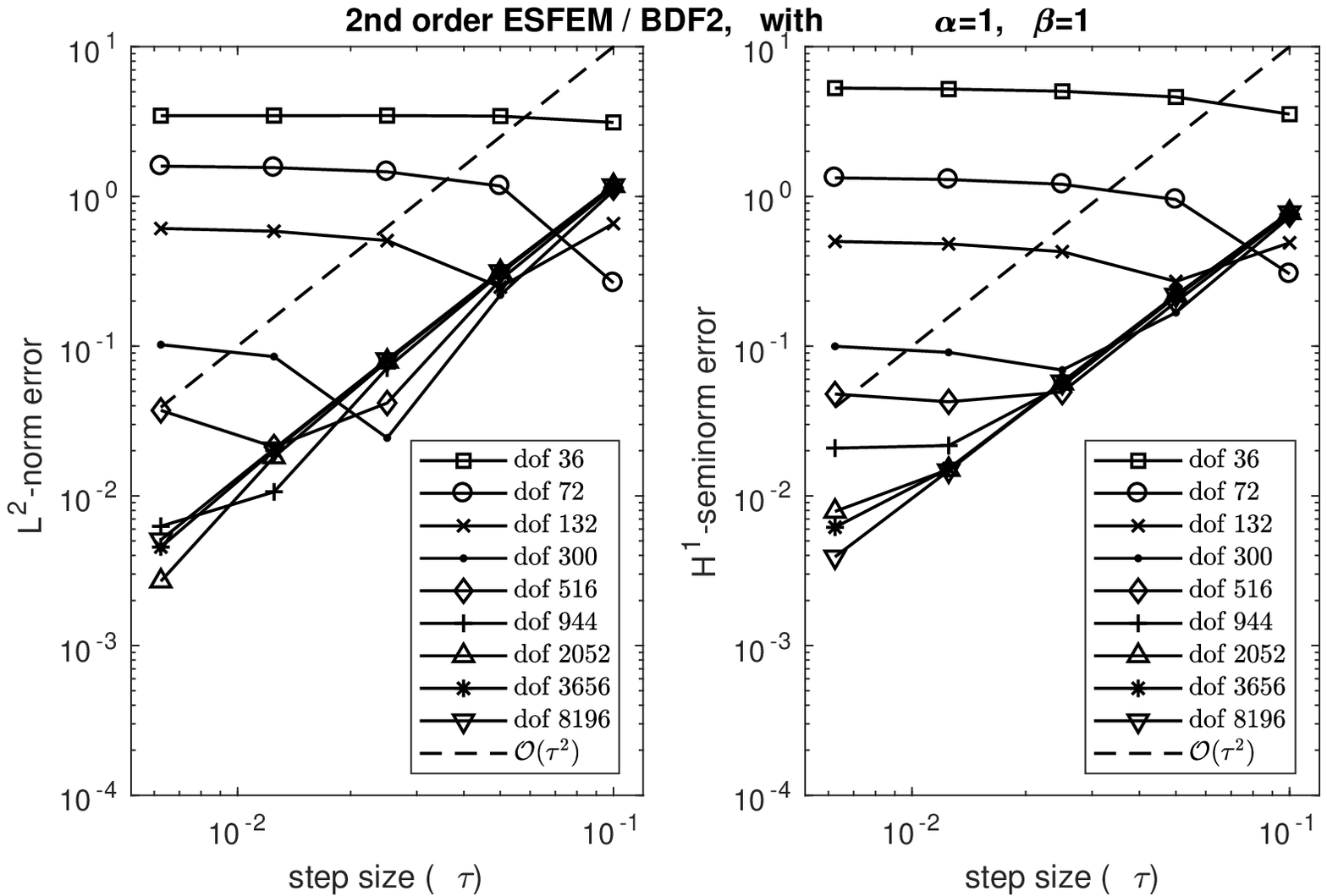}
	\caption{Temporal convergence of the BDF2 / quadratic ESFEM discretization for the surface-evolution equation \eqref{eq: numerics problem}}\label{fig: timeconv}
	\includegraphics[width=\textwidth,height=.42\textheight]{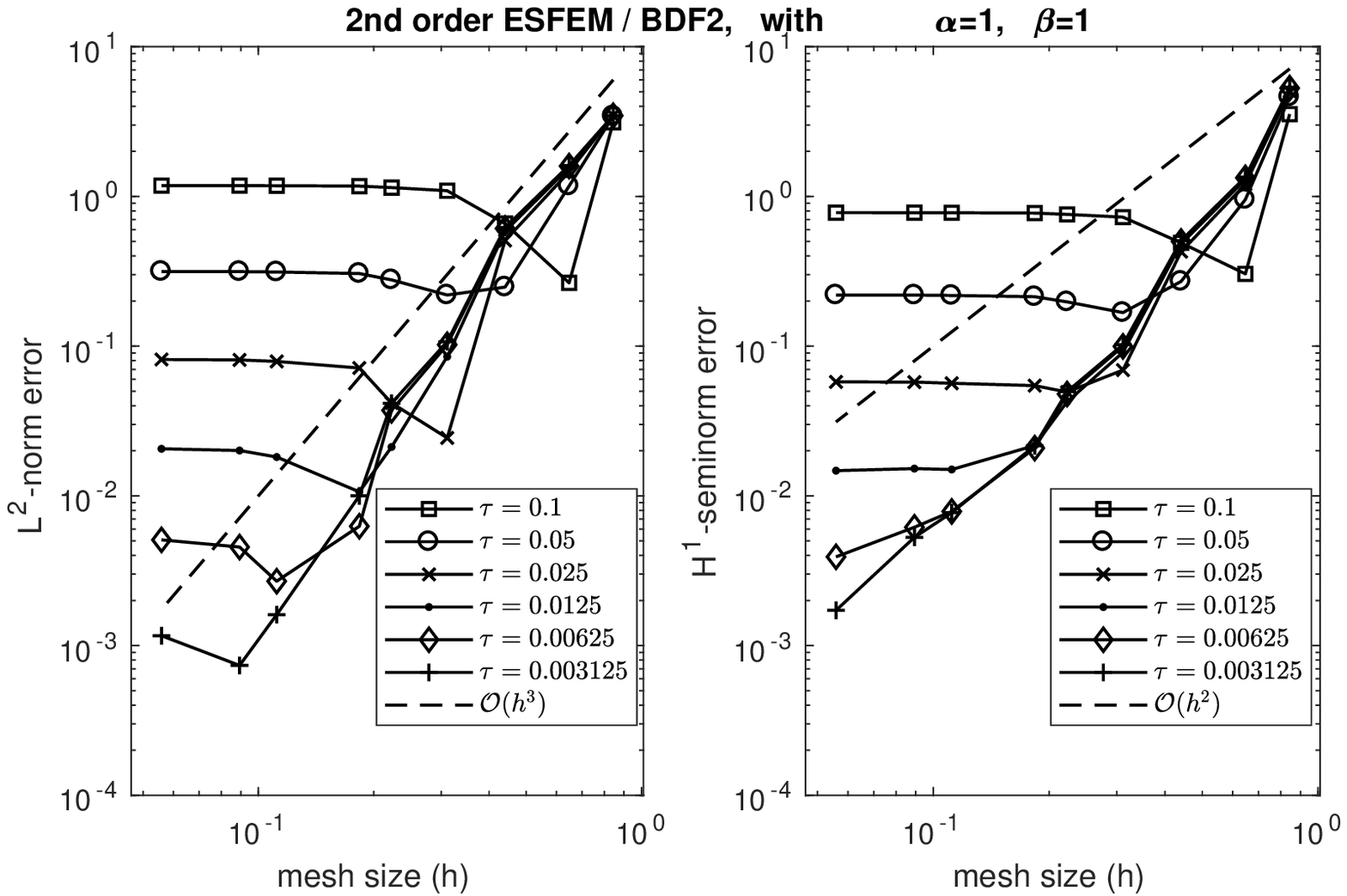}
	\caption{Spatial convergence of the BDF2 / quadratic ESFEM discretization for the surface-evolution equation \eqref{eq: numerics problem}}\label{fig: spaceconv}
\end{figure}

\clearpage

Figure~\ref{fig: BDF4_timeconv} 
shows the same errors for the BDF method of order~4. It is clearly seen that in this problem the BDF4 method gives much better accuracy than BDF2, at nearly the same computational cost.
\begin{figure}[htbp]
	\centering
	\includegraphics[width=\textwidth,height=.42\textheight]{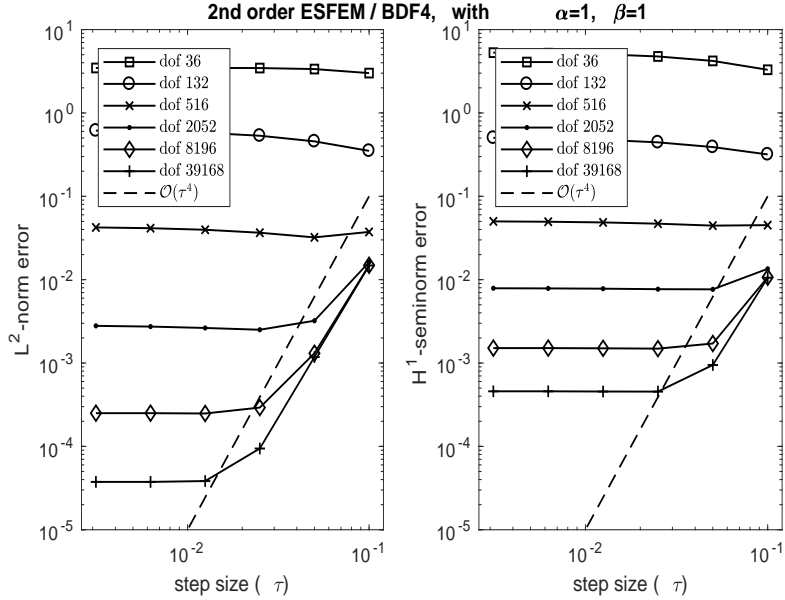}
	\caption{Temporal convergence of the BDF4 / quadratic ESFEM discretization for the surface-evolution equation \eqref{eq: numerics problem}}\label{fig: BDF4_timeconv}
	%
\end{figure}

\bbk
Numerical experiments for a semi-linear parabolic PDE system coupled to a velocity law on a surface with less symmetry, illustrating the coupled problem of Theorem~\ref{theorem: main-coupled}, are discussed in detail in our previous work \cite{soldriven}, where linearly implicit BDF methods have also been used.
\ebk
 
\bbk
\subsection{Mean curvature flow}

We also performed some numerical experiments, using mean curvature flow (MCF), to illustrate the effect of the elliptic regularisation. We again consider the problem \eqref{eq: numerics problem}, however without a forcing term, i.e.~the following form of mean curvature flow:
\begin{equation}
	\begin{aligned}
		v(x,t) - \alpha \laplace_{\Ga(t)} v(x,t) =&\ - \beta H_{\Ga(t)}(x)\, \nu_{\Ga(t)}(x) , \\
		\partial_t X(q,t) =&\ v(X(q,t),t) .
	\end{aligned}
\end{equation}
The initial surface is a rounded cube, the parameter $\beta$ is fixed to one.   Figure~\ref{fig:MCF compare} shows the results of different numerical experiments (using quadratic finite elements and BDF method of order $4$) at times $t=0, 0.2, 0.4, 0.5$ from top to bottom, while the parameter $\alpha$ is set to $0.1, 0.01, 0.001$ and $0$, from left to right, respectively. We note that our convergence results apply only to the case of a fixed positive $\alpha$, but the numerical experiments show good behaviour also for $\alpha\to 0$.

\newcommand{\subfloat}[1]{\begin{subfigure}[c]{0.24\textwidth}\includegraphics[width=\textwidth,height=.18\textheight]{#1}\end{subfigure}}

\begin{figure}[htbp]
	\centering
%
	\includegraphics[width=\textwidth,height=\textwidth]{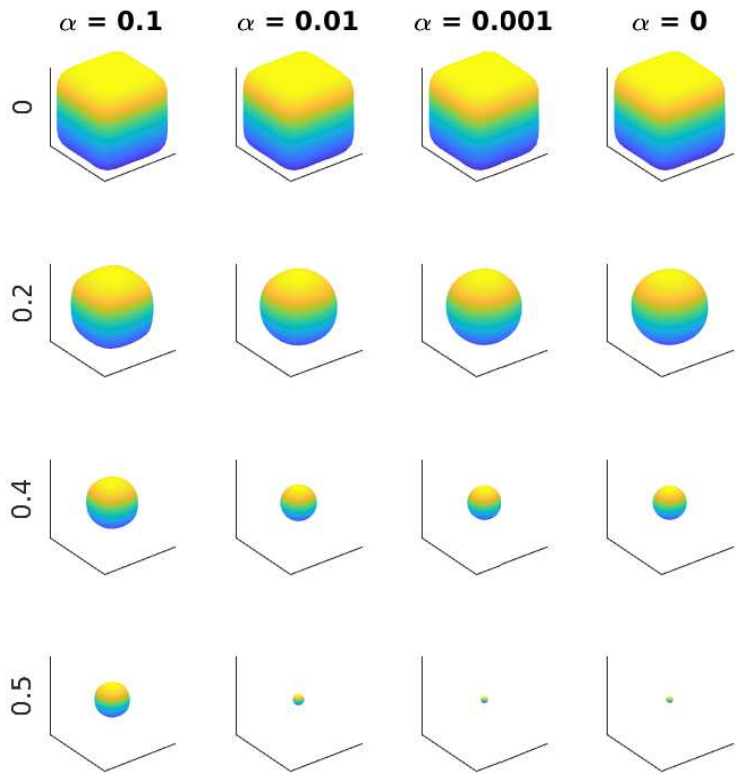}
	\caption{MCF with different values of $\alpha$ at different times}
	\label{fig:MCF compare}
\end{figure}


\ebk

\section*{Acknowledgement}
This work is supported by Deutsche Forschungsgemeinschaft, SFB 1173.


\end{document}